 \documentclass[12pt]{article} 
\usepackage{amsmath,amssymb,amsthm} %
\usepackage{mathtools}
\usepackage{esint}
\usepackage{graphicx}
\usepackage{color}
\usepackage{amsfonts}
\usepackage{epstopdf}
\usepackage{algorithm}
\usepackage{enumitem}
\setlist[itemize]{leftmargin=0.0mm}
\usepackage[noend]{algpseudocode}
\usepackage{diagbox}%slashbox
 \newtheorem{lemma}{Lemma}
 \newtheorem{theorem}{Theorem}[section]
\usepackage{footnote}
\newcommand{\norm}[1]{\left\|#1\right\|}

\newcommand{\rev}[1]{{\color{black}#1}}  
\graphicspath{ {./pic/} }
\usepackage{hyperref}

\begin{document}

\title{Multirate partially explicit scheme for multiscale flow problems}
% \author{Wing Tat Leung}\address{Department of Mathematics, University of California, Irvine, Irvine, CA 92697, USA (\email{wtleung@uci.edu}).}
 \author{Wing Tat Leung 
	\and Yating Wang } %\footnotemark{}
% \author{Wing Tat Leung \thanks{Department of Mathematics, University of California, Irvine, Irvine, CA 92697, USA (\email{wtleung@uci.edu}).}
% \and Yating Wang \thanks{Department of Mathematics, The University of Hong Kong, Pokfulam Road, Hong Kong SAR, China (\email{ytwang@hku.hk}). Corresponding author.}  } %\footnotemark{}
 \date{}
 \maketitle

\begin{abstract}
For time-dependent problems with high-contrast multiscale coefficients, the time step size for explicit methods is affected by the magnitude of the coefficient parameter. With a suitable construction of multiscale space, one can achieve a stable temporal splitting scheme where the time step size is independent of the contrast \cite{cem-splitting}. Consider the parabolic equation with heterogeneous diffusion parameter, the flow rates vary significantly in different regions due to the high-contrast features of the diffusivity. In this work, we aim to introduce a multirate partially explicit splitting scheme to achieve efficient simulation with the desired accuracy.
We first design multiscale subspaces to handle flow with different speeds. For the fast flow, we obtain a low-dimensional subspace for the high-diffusive component and adopt an implicit time discretization scheme. The other multiscale subspace will take care of the slow flow, and the corresponding degrees of freedom are treated explicitly. Then a multirate time stepping is introduced for the two parts. The stability of the multirate methods is analyzed for the partially explicit scheme. Moreover, we derive local error estimators corresponding to the two components of the solutions and provide an upper bound of the errors. An adaptive local temporal refinement framework is then proposed to achieve higher computational efficiency. Several numerical tests are presented to demonstrate the performance of the proposed method. %with respect to different combinations of time stepping

%A finer time scale can produce better approximations while the efficiency is guaranteed by . 

\end{abstract}

\section{Introduction}

	Modeling of flow and transport in complicated porous media in various physical and engineering applications encounters problems with multiscale features. In particular, the properties of the underlying media, such as thermal diffusivity or hydraulic conductivity, have values across different magnitudes. This poses challenges in the numerical simulation since the high contrast feature of the heterogeneous media introduces stiffness for the system. In terms of temporal discretization, the time-stepping depending on the magnitude of the multiscale coefficient is needed for explicit schemes. For the spatial discretization, \rev{multiscale methods including multiscale Finite Element Methods\cite{eh09, GMsFEM13}, variational multiscale method \cite{hfmq98}, heterogeneous multiscale methods\cite{ee03},  localized orthogonal decomposition \cite{maalqvist2014localization, henning2013oversampling, henning2014localized}, Gamblets\cite{owhadi2017multigrid} and many others} are introduced to handle the issue. %For the space discretization, due to the presence of the extremely small scale information,  fine meshes are needed which leads to a huge computational burden. In terms of temporal discretization, small time step size which depends on the maximum value of the multiscale coefficient is required for explicit schemes.  \cite{arbogast2007multiscale},
	 The multiscale model reduction methods include both local \cite{eh09, aarnesMMS, aarnes2008mixed, allaire2005multiscale, egw10} and global \cite{POD_notes, chinesta2011short, benner2015survey, nonlinear_deim, bui2008model} approaches to reduce computational expenses. The idea is to construct reduced order models to approximate the full fine-scale model and achieve efficient computation. Among these methodologies, the family of generalized multiscale finite element methods (GMsFEM) \cite{GMsFEM13, MixedGMsFEM, AdaptiveGMsFEM, adaptive_offline} are proposed to effectively address multiscale problem\rev{s} with high-contrast parameters.
	It first formulates some local problems on coarse grid regions to get snapshot bases that can capture the heterogeneous properties, and then designs appropriate spectral problems to get important modes in the snapshot space. The GMsFEM approach share\rev{s} some similarities with multi-continuum methods. The basis functions can recognize the high-contrast features such as channels that need to be represented individually. The convergence of the GMsFEM depends on the eigenvalue decay, and the small eigenvalues correspond to the high permeable channels. % The additional multiscale basis functions are constructed by selecting a few eigenfunctions from the spectral decomposition.

	To construct multiscale method such that the convergence is independent of the contrast and linearly decreases with respect to mesh size under suitable assumptions, the constraint energy minimizing GMsFEM (CEM-GMsFEM) was initiated\cite{cem-gmsfem, chung2019correction}. This approach begins with a suitable choice of auxiliary space, where some local spectral problems in coarse blocks are solved. The auxiliary space includes the minimal number of basis functions to identify the essential information of the channelized media. Then it will be used to compute the solutions of constraint energy minimizing problem in some oversampling coarse regions to handle the non-decaying property. The resulting localized solutions form the multiscale space.% and will be used to solve global problems.
	
	To adapt the CEM-GMsFEM for flow-based upscaling, the nonlocal multicontinuum upscaling method (NLMC) \cite{NLMC} is proposed by modifying the above framework. The idea is to use simplified auxiliary space by assuming that each separate fracture network within a coarse grid block is known. The auxiliary bases are piecewise constants corresponding to fracture networks and matrix, which are called continua. Then the local problems are formulated for each continuum by minimizing the local energy subject to appropriate constraints. This construction returns localized basis functions which can automatically identify each continuum. %These local solutions will then be adopted to compute the upscaled equation.
    Further, due to the property of the NLMC basis, this approach will provide non-local transmissibilities which describe the transfer among coarse blocks in an oversampled region and among different continua. %Furthermore, the coarse grid solutions from the upscaled equation are physically meaningful. nonlocal

% cem-splitting 
	Consider the time-dependent problem with high-contrast coefficients, there have been various approaches to handle multiscale stiff \rev{systems \cite{abdulle2012explicit, ariel2009multiscale, engquist2005heterogeneous, giraldo2013implicit, maalqvist2018multiscale, owhadi2017gamblets}}. Recently, a temporal splitting method is combined with the spatial multiscale method \cite{cem-splitting} to produce a contrast-independent partially explicit time discretization scheme. It splits the solution of the problem into two subspaces which can be computed using implicit and explicit methods, instead of splitting the operator of the equation directly based on physics \cite{sportisse2000analysis, tang1998convergence,verwer1998note, hundsdorfer2013numerical, marchuk1990splitting}. The multiscale subspaces are carefully constructed. The dominant basis functions stem from CEM-GMsFEM which have very few degrees of freedom and are treated implicitly. The additional space as a complement will be treated explicitly. It was shown that with the designed spaces, the proposed implicit-explicit scheme is unconditionally stable in the sense that the time step size is independent of the contrast. Following a similar idea in \cite{cem-splitting}, in this work, we will propose a multirate time-stepping method for the multiscale flow problem. 
	%Numerical tests indicate that the proposed partially explicit scheme produce similar results as the implicit method.
	
			% multirate 
    
	%explicit time discretization have proven effective in capturing the sharp, moving fronts common in these applications.  evolve each portion of the solution using time steps commensurate with the dynamics of that term
Multirate time integration method has been studied extensively in the past decades. Based on different splittings of the target equation, multiple time stepping is utilized in different parts of \rev{the system} according to computational cost or complexity of the physics. \rev{By partitioning the state variables into fast/active and slow/latent components, the multirate scheme with automatic step-size was introduced for linear multistep methods in \cite{gunther2001multirate}, and some self-adjusting multirate time stepping strategy was studied for stiff ODEs was discussed in \cite{savcenco2007multirate}. To handle the coupling between active and latent components and improve stability, schemes based on Runge–Kutta methods \cite{gunther2001multirate, gunther2016multirate} and Rosenbrock-Wanner methods \cite{gunther1993multirate} were proposed. In these approaches, the partition of the system is done in advance before performing a macro-step. To realize dynamic partitioning, multirate extrapolation methods were investigated \cite{engstler1997multirate}. Besides many applications, the multirate schemes were also favored in the simulation of PDEs including hyperbolic conservation laws and parabolic problems \cite{constantinescu2007multirate, dawson2001high, maurits1998explicit, osher1983numerical}. 
%Since the time step size needed to satisfy stability or convergence condition can change significantly over the course of a simulation,
The solutions of the parabolic equations may have some localized properties in space and time due to geometric features of the domain and boundaries, or the effects of the source term, thus adaptive time refinement schemes combined with local adaptivity in space are attractive approaches \cite{dawson1991finite, ewing1994finite, shishkin2000interpolation, trompert1991static}. In these works, some nested or composite grids were usually utilized. The difficulties that arise at the interface between local regions were treated carefully, and the time discretization was implicit or locally implicit. %Recently in \cite{abdulle2020explicit}, the authors devise an explicit multirate Runge–Kutta–Chebyshev method, and they proposed to use a modified equation such that the crippling effect of a few severely stiff components can be removed to enhance performance. 
There are many other multirate approaches to improve efficiency when solving multiscale parabolic problems \cite{abdulle2020explicit, carciopolo2020conservative, constantinescu2013extrapolated}. In this work, we split the solution of the parabolic equation into fast and slow components based on multiscale space construction and employ a partially explicit scheme to solve the splitting system with adaptive multirate time stepping.}  %For example, in convection-dominated flow problems, explicit time discretization is effective in capturing the sharp moving fronts, but the CFL condition hinders fast simulation. 

%overcome the stringent stability condition of standard methods without sacrificing explicitness

One key of our approach is to integrate the multirate approach with multiscale space construction. Due to the high contrast property of the coefficients, the solutions pass through different regions of the porous medium with different speeds in the flow problem. Different from the previous approach \cite{cem-splitting}, where the multiscale basis functions are formulated for dominant features (the first space) and complementary information (the second space), we propose to design multiscale spaces in different regions to handle the fast (the first space) and slow (the second space) components of the \rev{flow} separately. \rev{We remark that, in the previous approach, the problem can still be solved with the basis in the first space only, and the second space provides additional information to reduce the approximation error. However, in our approach, bases from both the first space and the second space are required to solve the problem.} We use the simplified auxiliary space containing piecewise constant functions as in the NLMC framework. We only keep the basis representing the high-diffusive region in the first space and adopt an implicit time discretization scheme. The second space consists of bases representing the remaining region, it will take care of the slow flow and the corresponding degrees of freedom are solved explicitly. Next, we introduce a multirate approach where different time step sizes are employed in the partially explicit splitting scheme, such that different parts of the solution are sought with time steps in line with the dynamics. We start with a coarse step size for both equations and refine local coarse time blocks based on some error estimators. With a finer discretization, the accuracy of the approximation can be improved. We analyze the stability of the multirate methods for all four cases when we use coarse or fine time step size alternatively for the implicit and explicit parts of the splitting scheme. It shows that the scheme is stable as long as the coarse time step size satisfies some suitable conditions independent of the contrast. 
	Moreover, we propose an adaptive algorithm for the splitting scheme by deriving error estimators based on the residuals. The two error estimators corresponding to the two components of the solutions can provide an upper bound of the errors. Compared with uniform refinement, an adaptive refining algorithm can enhance the efficiency significantly. Several numerical examples are presented to demonstrate the effectiveness of the proposed adaptive method.% implicit-explicit
    
    The paper is organized as follows. In Section \ref{sec:prelim}, we describe the problem setup and the partially explicit scheme. The construction of the multiscale spaces is discussed in Section \ref{sec:spaces}. In Section \ref{sec:main}, the multirate method is presented, the subsection \ref{sec:stability} is devoted to the stability analysis and the subsections \ref{sec:main_method}-\ref{sec:alg} present the adaptive algorithm. Numerical tests are shown in Section \ref{sec:numerical}. A conclusion is drawn in Section \ref{sec:conclusion}.

\section{Problem Setup}\label{sec:prelim}
Consider the parabolic equation 
\begin{equation*}
\begin{aligned}
    \frac{d u}{d t} - \nabla \cdot (\kappa \nabla u) &= f \;\;\;\; \text{ on } \Omega\times (0,T] \\
    u &= 0 \;\;\;\; \text{ on } \partial \Omega \times (0,T]\\
    u &= u_0  \;\;\;\; \text{ on } \partial \Omega \times \{0\}
\end{aligned}
\end{equation*}
where $\kappa \in L^{\infty}(\Omega)$ is a heterogeneous coefficient with high contrast, that is, the value of the conductivity/permeability in different regions of $\kappa$ can differ in magnitudes. 

The weak form of the problem is to seek $u(t,\cdot) \in V=H_0^1(\Omega)$ such that
\begin{equation*}
\begin{aligned}
    (\frac{\partial u }{\partial t}, v) + a(u,v) &= (f,v), \quad \forall v\in V, \quad t\in(0,T]\\
    u(0,\cdot) &= u_0 
\end{aligned}
\end{equation*}
where $\displaystyle{a(u,v) = \int_{\Omega} \kappa \nabla u \cdot \nabla v}$.

Now consider a coarse spatial partition $\mathcal{T}_H$ of the computational domain $\Omega$, we will construct suitable multiscale basis functions on $\mathcal{T}_H$ and form a multiscale space $V_H$ which is a subspace of $V$. Let $\tau$ be the time step size. The discretization in the space $V_H$ with implicit backward Euler scheme in time reads 
\begin{equation}\label{eq:standard_scheme}
    \left( \frac{u_H^{k+1} -u_H^k }{\tau}, v\right) + a(u_H^{k+1}, v) = (f^{k+1}, v), \;\;\;\; \forall v \in V_H
\end{equation} 
where $N=\frac{T}{\tau}$ is the number of time steps, $u_H^k = u_H(t_k)$\rev{, and $t_k = k\tau$}. It is well-known that this implicit scheme is unconditionally stable.
%, $t_{n} = n \tau$ with $n=0, \cdots, N-1$

Suppose the multiscale space $V_H$ can be decomposed into two subspaces
\begin{equation*}
    V_H = V_{H,1} + V_{H,2},
\end{equation*}
then a partial explicit temporal splitting scheme \cite{cem-splitting} is to find $u_{H,1}^{k} \in V_{H,1}$ and $u_{H,2}^{k} \in V_{H,2}$, for all $k$ satisfying
\begin{equation} \label{eq:partial_exp1}
    \left( \frac{u_{H,1}^{k+1} -u_{H,1}^k }{\tau}, v_1 \right) + \left( \frac{u_{H,2}^{k} -u_{H,2}^{k-1} }{\tau }, v_1 \right) + a(u_{H,1}^{k+1} +u_{H,2}^{k}, v_1) = (f^{k+1}, v_1),
\end{equation}
\begin{equation}
   \begin{aligned}\label{eq:partial_exp2}
        \left( \frac{u_{H,2}^{k+1} -u_{H,2}^k }{\tau}, v_2 \right) +\left( \frac{u_{H,1}^{k} -u_{H,1}^{k-1} }{\tau}, v_2 \right) + a((1-\omega)u_{H,1}^{k} + &\omega  u_{H,1}^{k+1} +u_{H,2}^{k}, v_2) \\
       &  = (f^{k+1}, v_2), %\frac{f^{k+1}+f^k}{2}
\end{aligned} 
\end{equation}
$\forall v_1 \in V_{H,1}, \forall v_2 \in V_{H,2}$, where $\omega \in [0,1]$ is a customized parameter. In the case $\omega = 0$, the two equations are decoupled, and can be solved simultaneously. In the case $\omega = 1$, the second equation depends on the solution $u_{H,1}^{k+1}$, thus the two equations will be solved sequentially.

The solution at time step $n+1$ will be $u_H^{n+1} = u_{H,1}^{n+1} + u_{H,2}^{n+1}$.
It was \rev{proved} in \cite{cem-splitting} that under appropriate choices of the multiscale spaces $V_{H,1}$ and $V_{H,2}$, the above implicit-explicit scheme resulted from the temporal splitting method for multiscale problems are stable with time step independent of contrast. In \cite{cem-splitting}, the dimension of $V_{H,1}$ is low and it contains some dominant multiscale basis functions, the second space $V_{H,2}$ includes additional bases representing the missing information. In this paper, we will construct multiscale spaces corresponding to different time scales, where the fast and slow parts of the solution are treated separately.

\section{Construction of multiscale spaces}\label{sec:spaces}
In this section, we will present the construction of multiscale spaces. 
% CEM NLMC
We will first discuss the basis construction for $V_{H,1}$ based on the contraint energy minimizing GMsFEM (CEM-GMsFEM) \cite{cem-gmsfem} and the nonlocal multicontinuum method (NLMC)\cite{NLMC, zhao2020analysis}.

\subsection{\rev{The idea of CEM-GMsFEM}}%for $V_{H,1}$}
To start with, we introduce some notations for the fine and coarse discretization of the computational domain $\Omega$. Let $\mathcal{T}^H$ 
be a coarse partition with mesh size $H$.%, and $\mathcal{T}^h$ be a conforming refinement of $\mathcal{T}^H$ with mesh size $h$, where $0 < h \ll H < 1$.
Denote by $\{K_i\}$ ($ i = 1, \cdots, N_c$) the set of coarse blocks 
in $\mathcal{T}^H$, and $K_i^+$ is an oversampled region with respect to each $K_i$, where the oversampling part contains a few layers of coarse blocks neighboring $K_i$. Let $V(K_i)$ be the restriction of $V=H_0^1(\Omega)$ on $K_i$.

Under the framework of CEM-GMsFEM, one first constructs an auxiliary space.  Consider the spectral problem
\begin{equation}\label{eq:spectral}
a_i(\phi_{\text{aux},k}^{(i)}, v) = \lambda_k^{i} s_i(\phi_{\text{aux},k}^{(i)}, v), \quad \forall v \in V(K_i), 
\end{equation}
where $\lambda_k^{i} \in \mathbb{R}$ and $\phi_{\text{aux},k}^{(i)} \in V(K_i)$ are corresponding eigenpairs, and 
\begin{equation*}
a_i(u,v) = \int_{K_i} \nabla u \cdot \nabla v, \quad \quad s_i(u,v) = \int_{K_i} \tilde{\kappa} uv,
\end{equation*}
with $\tilde{\kappa}  = \sum_{j} \kappa |\nabla \chi_j|^2$, and $\rev{ \chi_j}$ denotes the multiscale partition of unity function. Upon solving the spectral problem, we arrange the eigenvalues of \eqref{eq:spectral} in an ascending order, and select the first $l_i$ eigenfunctions to form the auxiliary basis functions. Define $V_{\text{aux}}^{(i)} :=  \text{span} \{\phi_{\text{aux},k}^{(i)}, \quad 1 \leq k \leq l_i\}$, where $1\leq i \leq N_c$ and $N_c$ is the number of coarse elements. Then the global auxiliary space $V_{\text{aux}} = \bigoplus_{i} V_{\text{aux}}^{(i)}$. We note that the auxiliary space needs to be chosen appropriately in order to get good approximation results. That is, the first few basis functions corresponding to small eigenvalues (representing all the channels) have to be included in the space. 

Define a projection operator $\pi_i: L^2(K_i) \mapsto V_{\text{aux}}^{(i)}$ \rev{as}
\begin{equation*}
   \pi_i (u) = \sum_{k=1}^{l_i} \frac{s_i(u,\phi_{\text{aux},k}^{(i)})}
   {s_i(\phi_{\text{aux},k}^{(i)},\phi_{\text{aux},k}^{(i)})} \phi_{\text{aux},k}^{(i)}, \;\; \forall u \in V,
\end{equation*}
and $\pi:  L^2(\Omega) \mapsto V_{\text{aux}}$ \rev{such that} $\pi =\sum_{i=1}^N \pi_i$. Define the null space of $\pi$ to be $\tilde{V}$:
$$\tilde{V} = \{v\in V\;|\;\; \pi(v)=0\}.$$

Let the global basis $\psi_{{glo},j}^{(i)}$ be the solution of the optimization problem
\begin{equation*}
 \begin{aligned}
    \psi_{{glo},j}^{(i)}  = \arg\!\min \{a(v,v)\; &| \; v\in V_0(K_i^+), \; s(v,\phi_{\text{aux},k}^{(i)}) = 1\; \\
    &\text{ and }s(v,\phi_{\text{aux},k'}^{(i')}) = 0 \; \; \forall i' \neq i, k'\neq k\},
\end{aligned}   
\end{equation*}
\rev{where $V_0(K_i^+)$ denotes the space of all functions in $V(K_i^+)$ with a vanishing trace on the boundary of $K_i^+$}. Define $V_{\text{glo}} = \text{span} \{\psi_{{glo},j}^{(i)},\; 1\leq i \leq N_c, \;  1\leq j \leq l_i \}$. 
It can be seen that $V_{\text{glo}}$ is $a$-orthogonal to $\tilde{V}$, that is 
\begin{equation*}
    a(\psi_{{glo},j}^{(i)}, v) = 0, \; \forall v \in \tilde{V}.
\end{equation*}

Then the CEM multiscale basis $\psi_{{cem},j}^{(i)}$ is a localization of $\psi_{{glo},j}^{(i)}$, and \rev{is} also computed using the auxiliary space $V_{\text{aux}}^{(i)}$. The idea is to solve the constraint energy minimization problem in a localized region $K_i^+$
\begin{equation}\label{eq:specbasis}
\begin{aligned}
 a(\psi_{{cem},j}^{(i)}, w) +  s(w, \mu_j^{(i)}) &= 0, \quad \forall w \in V(K_i^+),\\
s(\psi_{{cem},j}^{(i)}, \nu) &=  s(\phi_{\text{aux},j}^{(i)} , \nu),   \quad \forall \nu \in V_{\text{aux}}^{(\rev{i})}, 
\end{aligned}
\end{equation}
where $\phi_{\text{aux},j}^{(i)} \in  V_{\text{aux}}^{(i)}$ is an auxiliary basis.

The multiscale space is then $V_{cem} := \text{span} \{\psi_{{cem},j}^{(i)}, \; 1\leq j \leq  l_i, 1 \leq i \leq N_c\}$, it is an approximation to the global space $V_{\text{glo}}$. %can then be used to find the multiscale solution $u_{ms}$
% \begin{equation*}
%     (\frac{d u_{ms}}{d t}, v) + a(u_{ms}, v) = (f,v), \;\;\; \forall v \in V_{ms}.
% \end{equation*}

Note that the construction of CEM basis which we have presented here is general and can handle complex heterogeneous permeability field $\kappa$ (with high contrast). \rev{In this work, we assume $\kappa$ is a fractured media, where the value of $\kappa$ in the background region (called matrix) and in the fractured region are constants with high contrast, and the configuration of the highly permeable fractures in the domain is explicitly known. This assumption is reasonable in many real applications, thus we can consider a simplified construction of the basis functions in this case. } %In a fractured media where the configurations are explicitly known, here the assumption is valid in many real applications, we can consider the simplified construction. 
\subsection{\rev{Construction of multiscale spaces $V_{H,1}$ and $V_{H,2}$ based on NLMC}} \label{sec:V2}
The domain $\Omega$ for the media with fracture networks can be represented as follows
\begin{equation*}
\Omega = \Omega_m \bigoplus_{l=1}^s d_l \Omega_{f,l}
\end{equation*}
where the subscripts $m$ and $f$ denote the matrix and fractures correspondingly. In the fracture regions $\Omega_{f,l}$, the scalar $d_l$ denotes the aperture, and $s$ is the number of discrete fracture networks. The permeabilities of matrix and fractures usually \rev{differ in magnitudes}. 
%some simplified auxiliary basis can be adopted, 
In this setting, the constraint energy minimizing basis can be constructed via NLMC \cite{NLMC, zhao2020analysis} and the resulting basis functions can separate the continua such as matrix and fracture automatically. To be specific, for a given coarse block, we use constants for each separate fracture network, and then a constant for the matrix \rev{to form} the simplified auxiliary space. \rev{Specifically, for any coarse block $K_i$,  we write $K_i =K_{i,f} \cup K_{i,m}$ where $K_{i,f}$ is the high-contrast channelized region, and $K_{i,m}$ is its complement in $K_i$. Denote by $K_{i,f} =  \{ \Omega_{f,j} \cap K_i \neq \varnothing, \forall l =1, \cdots, s\}$ the set of discrete fractures/channels, we write $K_{i,f} :=  \{f_j^{(i)}, \; j = 1, \cdots, m_i\}$, and $m_i$ is the number of non-connected fractures in $K_i$. We then define two auxiliary spaces}
\rev{\begin{equation}\label{eq:nlmc_aux}
\begin{aligned}
V_{\text{aux},1}^{(i)} & = \text{span}\{\phi_{\text{aux},k}^{(i)}\; | \phi_{\text{aux},k}^{(i)} = 0 \text{ in } K_{i,m}, \; \phi_{\text{aux},k}^{(i)} = \delta_{jk} \text{ in } f_j^{(i)},  \;\; k =1, \cdots, m_i\}\\
V_{\text{aux},2}^{(i)} & = \text{span}\{\phi_{\text{aux},0}^{(i)}\; |
\phi_{\text{aux},0}^{(i)} = 1 \text{ in } K_{i,m}, \;
\phi_{\text{aux},0}^{(i)} = 0 \text{ in } K_{i,f}\}
\end{aligned}
\end{equation}
}
\rev{Consider an oversampled region $K_i^+$} of the coarse block $K_i$, \rev{following a similar idea as in CEM-GEMsFEM}, the NLMC basis $\psi_m^{(i)}$ are obtained by \rev{minimizing the energy $a(\psi_m^{(i)}, \psi_m^{(i)})$, with the constraints corresponding to the previously defined simplified auxiliary spaces \eqref{eq:nlmc_aux}. That is to find $\psi_m^{(i)} \in V_0(K_i^+)$ and $\mu_0^{(j)}, \mu_n^{(j)} \in \mathbb{R}$ from 
% \begin{align*}
%     \fint_{K_j} \psi_0^{(i)} = \delta_{ij}, \quad \fint_{f_l^{(i)}} \psi_0^{(i)} = 0,\\
%         \fint_{K_j} \psi_m^{(i)} = 0, \quad \fint_{f_l^{(j)}} \psi_m^{(i)} = \delta_{ij}\delta_{ml},
% \end{align*}
}the following localized constraint energy minimizing problem
\begin{equation}\label{eq:basis}
\begin{aligned}
&a(\psi_m^{(i)}, v) + \sum_{K_j \subset K_i^+} \left(\mu_0^{(j)} \int_{\rev{K_{j,m}}}v  + \sum_{1 \leq n \leq m_j} \mu_n^{(j)} \int_{ f_n^{(j)}} v \right) = 0, \quad   \forall v\in V_0(K_i^+), \\
&\int_{\rev{K_{j,m}}}\psi_m^{(i)}    = \delta_{ij} \delta_{m0}, \quad   \forall K_j \subset K_i^+, \\
&\int_{f_n^{(j)}} \psi_m^{(i)}    = \delta_{ij} \delta_{mn}, \quad   \forall f_n^{(j)} \in \mathcal{F}_{j}, \;  \forall K_j \subset K_i^+.
\end{aligned}
\end{equation}
  The NLMC basis \rev{functions are} then $ \{ \psi_m^{(i)}, \; 0 \leq m \leq  m_i, 1 \leq i \leq N_c\}$. We remark that the resulting basis separates the matrix and fractures automatically, and have spatial decay property\rev{\cite{cem-gmsfem, NLMC, zhao2020analysis}}. \rev{Furthermore, because the local auxiliary basis are constants within fractures and the matrix, the solution variables on the coarse level obtained using NLMC basis is physically meaningful, they denote the solution averages in each continuum (fracture/ channel) in each coarse region.}

One choice of the two multiscale spaces is to let $V_{H,1} = \text{span}\{ \psi_m^{(i)}, \; 1 \leq m \leq  m_i, 1 \leq i \leq N_c\}$ and $V_{H,2} = \text{span}\{ \psi_0^{(i)}, \; 1 \leq i \leq N_c\}$. In this work, we further want to include the constant basis in the second space $V_{H,2}$. Thus we perform an additional step to slightly modify the definition of two spaces. Denote \rev{the average of all NLMC basis by}%representing the fractures by
\begin{equation}\label{eq:average_basis}
    \bar{\psi}:= \frac{1}{L} \sum_{i=1}^{N_c} \sum_{m=\rev{0}}^{m_i} \psi_m^{(i)},
\end{equation}
where $L=\sum_{i=1}^{N} m_i$. 

Let $\displaystyle{ \tilde{\psi}_m^{(i)} =\psi_m^{(i)} - \frac{s(\psi_m^{(i)}, \bar{\psi})}{s(\bar{\psi},\bar{\psi})}\bar{\psi}},\; 0 \leq m \leq  m_i, \; 1 \leq i \leq N_c$. %\rev{Note that, the subscript $m$ starts from $1$ in the definition $\tilde{\psi}_m^{(i)}$, which indicates the basis functions corresponding to the matrix are not included here.}
To simplify the notation, we omit the double scripts in $\tilde{\psi}_m^{(i)}$ and denote the set of basis by $\{\displaystyle{ \tilde{\psi}_k, \;\; k = 1, \cdots, L} \}$.

Finally, we define the space $V_{H,1}$ as follows:
\begin{equation}\label{eq:VH1}
    V_{H,1} = \text{span} \{\tilde{\psi}_k, \quad 1 \leq k \leq L-1\}.
\end{equation}
The basis functions corresponding to the matrix and the basis \rev{$\bar{\psi}$ will be included in the second subspace $V_{H,2}$, that is 
\begin{equation}\label{eq:VH2}
     V_{H,2} = \text{span} \{ \bar{\psi},\; \psi_0^{(i)},\; 1 \leq i \leq N_c\}.
\end{equation}
}
\rev{We note that we take away the last basis in $V_{H,1}$} to remove linear dependency \rev{between the two spaces}. \rev{By this construction, $V_{H,1}$ contains basis representing the high contrast fractures/channels only, and $V_{H,2}$ contains basis representing the background matrix and the constant basis. This separates the slow and fast flow regions of the media.}

% \rev{\begin{theorem}
% \end{theorem}}

\rev{
In this work, the simplified basis construction works well for the fractured media. We remark that in heterogeneous media, the spaces $V_{H,1}$ and $V_{H,2}$ can be enriched to enhance the approximation of the solutions. The spatial enrichment will be investigated in our future work.}

\section{Multirate time stepping for partially explicit scheme}\label{sec:main}

Based on the multiscale spaces constructed in Section \ref{sec:spaces}, we introduce a multirate time stepping partially explicit temporal splitting scheme. Consider the coarse time step size $\Delta T$ and fine time step size $\Delta t$, where $\Delta T = m \Delta t$.
Denote by the fine partition of the time domain $(0,T]$ by 
$$ 0=t_0 < t_1 < \cdots<  t_{N-1} = T.$$
The coarse partition of the time domain $(0,T]$ is formed by $$ 0=T_0< T_1<  \cdots<  T_{(N-1)/m} = T.$$
Further, we write each coarse time interval $(T_k, T_{k+1}] = \cup_{j=n_k}^{n_{k+1}-1} (t_{j}, t_{j+1}]$ where $n_k=km$.

The multirate scheme is then defined as follows. \rev{In} each coarse interval $(T_k, T_{k+1}]$, we are seeking for $u^{n_{k+1}}=u_1^{n_{k+1}}+u_2^{n_{k+1}}$ \rev{given the solution at the previous coarse time step} $u^{n_{k}}=u_1^{n_{k}}+u_2^{n_{k}}$. \rev{The two equations will take the time steps in the following four cases:} using coarse time step size in both \eqref{eq:partial_exp1} and \eqref{eq:partial_exp2} (coarse-coarse), using coarse time step size in \eqref{eq:partial_exp1} and using fine time step size in  \eqref{eq:partial_exp2} (coarse-fine), using coarse time step size in \eqref{eq:partial_exp1} and using fine time step size in  \eqref{eq:partial_exp2} (fine-coarse), using fine time step size in \eqref{eq:partial_exp1} and using fine time step size in both \eqref{eq:partial_exp1} and \eqref{eq:partial_exp2} (fine-fine).

\begin{itemize}
    \item[]\textbf{Case 1 (coarse-coarse)}: Coarse time step size for \eqref{eq:partial_exp1}, coarse time step size for \eqref{eq:partial_exp2}. That is, take $\tau = \Delta T$ in both equations, let $\bar{u}_{H,1}^{n_{k+1}} = (1-\omega)u_{H,1}^{n_k} +\omega  u_{H,1}^{n_{k+1}}$: 
    \begin{equation}\label{eq:partial_exp_case1}
    \begin{aligned}
     &\left( \frac{u_{H,1}^{n_{k+1}} -u_{H,1}^{n_k} }{\Delta T}, v_1 \right) + \left( \frac{u_{H,2}^{n_{k}} -u_{H,2}^{n_{k-1}} }{\Delta T}, v_1 \right) + a(u_{H,1}^{n_{k+1}} +u_{H,2}^{n_k}, v_1) = 0, \\
     &\left( \frac{u_{H,2}^{n_{k+1}} -u_{H,2}^{n_k} }{\Delta T}, v_2 \right) +\left( \frac{u_{H,1}^{n_k} -u_{H,1}^{n_{k-1}} }{\Delta T}, v_2 \right) + a( \bar{u}_{H,1}^{n_{k+1}}+u_{H,2}^{\rev{n_k}}, v_2) = 0,  
     \end{aligned}
    \end{equation}
    $\forall v_1 \in V_{H,1}, \forall v_2 \in V_{H,2}$.
    
    \item[]\textbf{Case 2 (coarse-fine)}: Coarse time step size for \eqref{eq:partial_exp1}, fine time step size for \eqref{eq:partial_exp2}, let $\bar{u}_{H,1}^{n_{k+1}} = (1-\omega)u_{H,1}^{n_k} +\omega  u_{H,1}^{n_{k+1}}$: 
    \begin{equation}\label{eq:partial_exp_case2}
    \begin{aligned}
     &\left( \frac{u_{H,1}^{n_{k+1}} -u_{H,1}^{n_k} }{\Delta T}, v_1 \right) + \left( \frac{u_{H,2}^{n_{k}} -u_{H,2}^{n_{k-1}} }{\Delta T}, v_1 \right) + a(u_{H,1}^{n_{k+1}} +u_{H,2}^{n_k}, v_1) = 0, \\
     &\left( \frac{u_{H,2}^{n+1} -u_{H,2}^n }{\Delta t}, v_2 \right) +\left( \frac{u_{H,1}^{n_k} -u_{H,1}^{n_{k-1}} }{\Delta T}, v_2 \right) + a( \bar{u}_{H,1}^{n_{k+1}}+u_{H,2}^{n}, v_2) = 0,  
     \end{aligned}
    \end{equation}
    $\forall v_1 \in V_{H,1}, \forall v_2 \in V_{H,2}$, and for $n = n_k, n_k+1, \cdots, n_{k+1}-1$.

    \item[]\textbf{Case 3 (fine-coarse)}: Fine time step size for \eqref{eq:partial_exp1}, coarse time step size for \eqref{eq:partial_exp2} 
    \begin{equation}\label{eq:partial_exp_case3}
    \begin{aligned}
     &\left( \frac{u_{H,1}^{n+1} -u_{H,1}^{n} }{\Delta t}, v_1 \right) + \left( \frac{u_{H,2}^{n_{k}} -u_{H,2}^{n_{k-1}} }{\Delta T}, v_1 \right) + a(u_{H,1}^{n+1} +u_{H,2}^{n_k}, v_1) = 0, \\
     &\left( \frac{u_{H,2}^{n_{k+1}} -u_{H,2}^{n_k} }{\Delta T}, v_2 \right) +\left( \frac{u_{H,1}^{n_k} -u_{H,1}^{n_{k-1}} }{\Delta T}, v_2 \right) + a(\bar{u}_{H,1}^{n_{k+1}}+u_{H,2}^{n_k}, v_2) = 0, 
     \end{aligned}
    \end{equation}
     $\forall v_1 \in V_{H,1}, \forall v_2 \in V_{H,2}$, and for $n = n_k, n_k+1, \cdots, n_{k+1}-1$.

    \item[]\textbf{Case 4 (fine-fine)}: Fine time step for \eqref{eq:partial_exp1}, fine time step for \eqref{eq:partial_exp2}. That is, take $\tau = \Delta t$ in both equations, let $\bar{u}_{H,1}^{n+1} = (1-\omega)u_{H,1}^{n} +\omega  u_{H,1}^{n+1}$: 
    \begin{equation}\label{eq:partial_exp_case4}
    \begin{aligned}
     &\left( \frac{u_{H,1}^{n+1} -u_{H,1}^{n} }{\Delta t}, v_1 \right) + \left( \frac{u_{H,2}^{n} -u_{H,2}^{n-1} }{\Delta \rev{t} }, v_1 \right) + a(u_{H,1}^{n+1} +u_{H,2}^{n}, v_1) = 0, \\
     &\left( \frac{u_{H,2}^{n+1} -u_{H,2}^{n} }{\Delta \rev{t}}, v_2 \right) +\left( \frac{u_{H,1}^{n} -u_{H,1}^{n-1} }{\Delta \rev{t}}, v_2 \right) + a( \bar{u}_{H,1}^{n+1}+u_{H,2}^{n}, v_2) = 0,  
     \end{aligned}
    \end{equation}
    $\forall v_1 \in V_{H,1}, \forall v_2 \in V_{H,2}$, and for $n = n_k, n_k+1, \cdots, n_{k+1}-1$
\end{itemize}
We remark that \rev{in the global scheme, since the above four cases may occur alternatively, if cases 1-3 are chosen in one coarse time interval and case 4 is chosen in the following interval, $u_{H,i}^{n}$ will not be defined at the fine time steps in the previous macro-step.} In this case, we use the linear interpolation of the nearest two coarse time step solutions $u_{H,i}^{n_{k+1}},u_{H,i}^{n_{k}}$ to define intermediate time step solutions  $u_{H,i}^{n}$ for $n_{k}<n<n_{k+1}$.

\subsection{Stability for different cases}\label{sec:stability}
Consider a coarse time block $(T_k, T_{k+1}]$, the stability of the multirate method for the above mentioned four cases is proved in this subsection.

Let $\gamma$ be a constant such that 
\begin{equation}\label{eq:gamma}
    \gamma =  \sup_{v_1 \in V_{H,1}, v_2 \in V_{H,2}} \frac{(v_1, v_2)}{\|v_1\|\|v_2\|} < 1.
\end{equation}
\rev{We remark that $\gamma$ can be estimated from the given spaces. }

For case 1 and case 4 defined in section \ref{sec:alg}, following a similar proof in \cite{cem-splitting}, the partially explicit scheme \eqref{eq:partial_exp1}-\eqref{eq:partial_exp2} is stable if 
\begin{equation*}
    \tau \sup_{v\in V_{H,2}} \frac{\norm{v}^2_a}{\norm{v}^2} \leq \frac{1-\gamma^2}{2-\omega},
\end{equation*}
and $\tau = \Delta T$ for case 1, $\tau = \Delta t$ for case 4. 

We will show the stability for case 2 and 3 in the following.

\subsubsection{Stability for case 2}

Use the coarse time step for $u_{H,1}$ and use the fine time step for $u_{H,2}$,

\begin{lemma}\label{lemma1}
The multirate partially explicit scheme in \eqref{eq:partial_exp_case2} satisfies the stability estimate
\begin{equation*}
    \frac{\gamma^2\Delta T}{2}\sum_{j=1}^2 \|  \frac{u_{H,j}^{n_{k+1}} -u_{H,j}^{n_k}}{\Delta T}\|^2 + \frac{1}{2}\|u_{H}^{n_{k+1}}\|_a^2 
   \leq  \frac{\gamma^2\Delta T}{2}\sum_{j=1}^2 \|  \frac{u_{H,j}^{n_{k}} -u_{H,j}^{n_{k-1}}}{\Delta T}\|^2 +\frac{1}{2}  \|u_{H}^{n_k}\|_a^2. 
\end{equation*}
if 
\begin{equation}\label{eq:stablility_case2}
    \Delta T \sup_{v\in V_{H,2}} \frac{ \|v\|_a^2}{\|v\|^2} \leq \frac{(1-\gamma^2)m}{m+1-m\omega}.
\end{equation}
\end{lemma}

\begin{proof}
The equations in \eqref{eq:partial_exp_case2}  can be written as
\begin{equation}\label{eq:partial_exp_case2_1}
 \left( u_{H,1}^{n_{k+1}} -u_{H,1}^{n_k} + u_{H,2}^{n_{k}} -u_{H,2}^{n_{k-1}}, v_1 \right) =- \Delta T a(u_{H,1}^{n_{k+1}} +u_{H,2}^{n_k}, v_1),
\end{equation}
\begin{equation}\label{eq:partial_exp_case2_2}
\left( m(u_{H,2}^{n+1} -u_{H,2}^n)+u_{H,1}^{n_k} -u_{H,1}^{n_{k-1}} , v_2 \right)  =-  \Delta T a((1-\omega)u_{H,1}^{n_k} +\omega  u_{H,1}^{n_{k+1}} +u_{H,2}^{n}, v_2). 
\end{equation}

Take $v_1 = u_{H,1}^{n_{k+1}} -u_{H,1}^{n_k}$ in \eqref{eq:partial_exp_case2_1}, take $v_2 = u_{H,2}^{n+1} -u_{H,2}^{n}$ in \eqref{eq:partial_exp_case2_2} and sum over $n=n_k, n_k+1, \cdots, n_{k+1}-1$. Then for the left hand side of \eqref{eq:partial_exp_case2_1}, we get
\begin{equation*}
  \begin{aligned}
   &\left( u_{H,1}^{n_{k+1}} -u_{H,1}^{n_k} + u_{H,2}^{n_{k}} -u_{H,2}^{n_{k-1}},  u_{H,1}^{n_{k+1}} -u_{H,1}^{n_k} \right) \\
%   &= \|  u_{H,1}^{n_{k+1}} -u_{H,1}^{n_k}\|^2 + \left(u_{H,2}^{n_{k}} -u_{H,2}^{n_{k-1}},  u_{H,1}^{n_{k+1}} -u_{H,1}^{n_k} \right) \\
 & \geq  \|  u_{H,1}^{n_{k+1}} -u_{H,1}^{n_k}\|^2 -\gamma  \|u_{H,2}^{n_{k}} -u_{H,2}^{n_{k-1}}\| \| u_{H,1}^{n_{k+1}} -u_{H,1}^{n_k} \| \\
  & \geq \frac{1}{2} \|  u_{H,1}^{n_{k+1}} -u_{H,1}^{n_k}\|^2 -\frac{\gamma^2}{2} \|u_{H,2}^{n_{k}} -u_{H,2}^{n_{k-1}}\|^2
\end{aligned}  
\end{equation*}

For the left hand side of  \eqref{eq:partial_exp_case2_2}, we have
\begin{equation*}
\begin{aligned}
&\sum_{n=n_k}^{n_{k+1}-1}\left( m(u_{H,2}^{n+1} -u_{H,2}^n)+u_{H,1}^{n_k} -u_{H,1}^{n_{k-1}} ,  u_{H,2}^{n+1} -u_{H,2}^{n} \right) \\
% &=  \sum_{n=n_k}^{n_{k+1}-1} \|m(u_{H,2}^{n+1} -u_{H,2}^n)\|^2 + \left( u_{H,1}^{n_k} -u_{H,1}^{n_{k-1}} ,  u_{H,2}^{n_{k+1}} -u_{H,2}^{n_k} \right) \\
&\geq   \sum_{n=n_k}^{n_{k+1}-1} \rev{m}\|u_{H,2}^{n+1} -u_{H,2}^n\|^2 - \frac{\gamma^2}{2} \|\rev{u_{H,1}^{n_{k}} -u_{H,1}^{n_{k-1}} } \|^2 - \frac{1}{2} \| \rev{u_{H,2}^{n_{k+1}} -u_{H,2}^{n_{k}}}\|^2 \\
&\geq  \frac{m}{2} \sum_{n=n_k}^{n_{k+1}-1} \|u_{H,2}^{n+1} -u_{H,2}^n\|^2 - \frac{\gamma^2}{2} \|\rev{u_{H,1}^{n_{k}} -u_{H,1}^{n_{k-1}} }\|^2
\end{aligned}    
\end{equation*}
since $\displaystyle{  - \frac{1}{2}\|  \rev{u_{H,2}^{n_{k+1}} -u_{H,2}^{n_{k}}}\|^2 \geq  - \frac{m}{2} \sum_{n=n_k}^{n_{k+1}-1} \|(u_{H,2}^{n+1} -u_{H,2}^n)\|^2 }$.

Sum up the right hand side of \eqref{eq:partial_exp_case2_1} and  \eqref{eq:partial_exp_case2_2}, we have
\begin{equation}\label{eq:case2_ineq1}
\begin{aligned}
   & - \Delta T a(u_{H,1}^{n_{k+1}} +u_{H,2}^{n_k}, u_{H,1}^{n_{k+1}} -u_{H,1}^{n_k} )  - (1-\omega) \Delta T a(u_{H,1}^{n_k}, u_{H,2}^{n_{k+1}} -u_{H,2}^{n_{k}} )\\
    &- \omega \Delta T a(u_{H,1}^{n_{k+1}} , u_{H,2}^{n_{k+1}} -u_{H,2}^{n_{k}} ) -\Delta T  \sum_{n=n_k}^{n_{k+1}-1} \rev{a}( u_{H,2}^{n}, u_{H,2}^{n+1} -u_{H,2}^{n}) \\
    =&   - \Delta T a(u_{H,1}^{n_{k+1}}, u_{H,1}^{n_{k+1}} -u_{H,1}^{n_k} )     +\Delta T a(u_{H,2}^{n_k}, u_{H,1}^{n_k} )   - \Delta T a(u_{H,1}^{n_{k+1}}, u_{H,2}^{n_{k+1}} ) \\
& + (1-\omega) \Delta T a(u_{H,1}^{n_{k+1}}-u_{H,1}^{n_k}, u_{H,2}^{n_{k+1}} -u_{H,2}^{n_{k}} ) -\Delta T  \sum_{n=n_k}^{n_{k+1}-1} \rev{a}( u_{H,2}^{n}, u_{H,2}^{n+1} -u_{H,2}^{n}) \\
& =: \text{RHS}
\end{aligned}
\end{equation}
Note that for the terms in RHS in the above inequalities, we have 
\begin{equation*}
\begin{aligned}
    & - a(u_{H,1}^{n_{k+1}}, u_{H,1}^{n_{k+1}} -u_{H,1}^{n_k} ) \rev{=} -\frac{ 1}{2} \left( \|u_{H,1}^{n_{k+1}}\|_a^2 + \|  u_{H,1}^{n_{k+1}} -u_{H,1}^{n_k}\|_a^2 - \|u_{H,1}^{n_k}\|_a^2 \right), \\
     &  \sum_{n=n_k}^{n_{k+1}-1} \rev{a} ( u_{H,2}^{n}, u_{H,2}^{n+1} -u_{H,2}^{n}) \rev{=} - \frac{1}{2}  \sum_{n=n_k}^{n_{k+1}-1} \left( \|u_{H,2}^{n}\|_a^2 + \|  u_{H,2}^{n+1} -u_{H,2}^{n}\|_a^2 - \|u_{H,2}^{n+1}\|_a^2 \right)\\
     & \hspace{12em} = - \frac{1}{2}  \left(\|u_{H,2}^{n_{k}}\|_a^2 - \|u_{H,2}^{n_{k+1}}\|_a^2  + \sum_{n=n_k}^{n_{k+1}-1}  \|  u_{H,2}^{n+1} -u_{H,2}^{n}\|_a^2  \right), \\
    &   a(u_{H,1}^{n_{k+1}}-u_{H,1}^{n_k}, u_{H,2}^{n_{k+1}} -u_{H,2}^{n_{k}} ) \leq \frac{1}{2}  \left(  \|  u_{H,1}^{n_{k+1}} -u_{H,1}^{n_k}\|_a^2 + \|  u_{H,2}^{n_{k+1}} -u_{H,2}^{n_k}\|_a^2 \right).
\end{aligned}
\end{equation*}
Substitute these into the left of \eqref{eq:case2_ineq1} and regroup terms, we get
\begin{equation*}
    \begin{aligned}
         \text{RHS}  & \leq -\frac{  \Delta T}{2}\|u_{H}^{n_{k+1}}\|_a^2  + \frac{  \Delta T}{2}  \|u_{H}^{n_k}\|_a^2 +  \frac{\Delta T}{2} \sum_{n=n_k}^{n_{k+1}-1}\|  u_{H,2}^{n+1} -u_{H,2}^{n}\|_a^2 \\
  & - \frac{\omega \Delta T}{2}   \|  u_{H,1}^{n_{k+1}} -u_{H,1}^{n_k}\|_a^2 + \frac{(1-\omega) \Delta T}{2}  \|  u_{H,2}^{n_{k+1}} -u_{H,2}^{n_k}\|_a^2. 
    \end{aligned}
\end{equation*}

Combine the results, we have 
\begin{equation*}
\begin{aligned}
    &\frac{\gamma^2}{2}\sum_{j=1}^2 \|  u_{H,j}^{n_{k+1}} -u_{H,j}^{n_k}\|^2 + \frac{1-\gamma^2}{2} \|  u_{H,1}^{n_{k+1}} -u_{H,1}^{n_k}\|^2  + \frac{m (1-\gamma^2)}{2}  \sum_{n=n_k}^{n_{k+1}-1} \|u_{H,2}^{n+1} -u_{H,2}^n\|^2 \\
   & + \frac{  \Delta T}{2}\|u_{H}^{n_{k+1}}\|_a^2  \leq  \frac{\Delta T (m+1-m\omega)}{2} \sum_{n=n_k}^{n_{k+1}-1}\|  u_{H,2}^{n+1} -u_{H,2}^{n}\|_a^2  +\frac{\Delta T}{2}  \|u_{H}^{n_k}\|_a^2  \\
   & + \frac{\gamma^2}{2}\sum_{j=1}^2 \|  u_{H,j}^{n_{k}} -u_{H,j}^{n_{k-1}}\|^2
\end{aligned}    
\end{equation*}
where we use the fact that $\displaystyle{ \|  u_{H,2}^{n_{k+1}} -u_{H,2}^{n_k}\|_a^2  \leq  m  \sum_{n=n_k}^{n_{k+1}-1}\|  u_{H,2}^{n+1} -u_{H,2}^{n}\|_a^2 }$.

As long as 
\begin{equation*}
    \frac{\Delta T (m+1-m\omega)}{2} \sum_{n=n_k}^{n_{k+1}-1}\|  u_{H,2}^{n+1} -u_{H,2}^{n}\|_a^2 \leq \frac{m (1-\gamma^2)}{2}  \sum_{n=n_k}^{n_{k+1}-1} \|u_{H,2}^{n+1} -u_{H,2}^n\|^2
\end{equation*}
we have
\begin{equation*}
     \frac{\gamma^2}{2}\sum_{j=1}^2 \|  u_{H,j}^{n_{k+1}} -u_{H,j}^{n_k}\|^2 + \frac{  \Delta T}{2}\|u_{H}^{n_{k+1}}\|_a^2 \leq \frac{\gamma^2}{2}\sum_{j=1}^2 \|  u_{H,j}^{n_{k}} -u_{H,j}^{n_{k-1}}\|^2 +\frac{  \Delta T}{2}  \|u_{H}^{n_k}\|_a^2.
\end{equation*}

Thus the stability condition is 
\begin{equation*}
    \Delta T \sup_{v\in V_{H,2}} \frac{ \|v\|_a^2}{\|v\|^2} \leq \frac{(1-\gamma^2)m}{m+1-m\omega}.
\end{equation*}
%The last step is ensured by the stability condition \eqref{eq:stablility_case2}.
\end{proof}

We remark that, the stability condition becomes $\displaystyle{\Delta t  \sup_{v\in V_{H,2}} \frac{\|v\|_a^2}{\|v\|^2} \leq (1-\gamma^2) }$ if $\omega = 1$, which means we only need the fine time step size (for the explicit part) to satisfy the condition to ensure stability.

\subsubsection{Stability for case 3}

\begin{lemma}\label{lemma2}
The multirate partially explicit scheme in \eqref{eq:partial_exp_case3} satisfies the stability estimate
\begin{equation*}
    \frac{\gamma^2\Delta T}{2}\sum_{j=1}^2 \|  \frac{u_{H,j}^{n_{k+1}} -u_{H,j}^{n_k}}{\Delta T}\|^2 + \frac{1}{2}\|u_{H}^{n_{k+1}}\|_a^2 
   \leq  \frac{\gamma^2\Delta T}{2}\sum_{j=1}^2 \|  \frac{u_{H,j}^{n_{k}} -u_{H,j}^{n_{k-1}}}{\Delta T}\|^2 +\frac{1}{2}  \|u_{H}^{n_k}\|_a^2. 
\end{equation*}
if  
\begin{equation}\label{eq:stablility_case3}
   \Delta T  \sup_{v\in V_{H,2}} \frac{\|v\|_a^2}{\|v\|^2} \leq \frac{(1-\gamma^2)}{m-m\omega+1}.
\end{equation}

\end{lemma}

\begin{proof}
The equations in \eqref{eq:partial_exp_case3}  can be written as
\begin{equation}\label{eq:partial_exp_case3_1}
 \left( m(u_{H,1}^{n+1} -u_{H,1}^{n}) + u_{H,2}^{n_{k}} -u_{H,2}^{n_{k-1}}, v_1 \right) =- \Delta T a(u_{H,1}^{n+1} +u_{H,2}^{n_k}, v_1),
\end{equation}
\begin{equation}\label{eq:partial_exp_case3_2}
\left( u_{H,2}^{n_{k+1}} -u_{H,2}^{n_k} +u_{H,1}^{n_k} -u_{H,1}^{n_{k-1}} , v_2 \right)  =-  \Delta T a((1-\omega)u_{H,1}^{n_k} +\omega  u_{H,1}^{n_{k+1}} +u_{H,2}^{n_k}, v_2),  
\end{equation}

Take $v_1 = u_{H,1}^{n+1} -u_{H,1}^{n}$ in \eqref{eq:partial_exp_case3_1} and sum over $n=n_k, n_k+1, \cdots, n_{k+1}-1$, also take $v_2 = u_{H,2}^{n_{k+1}} -u_{H,2}^{n_k}$ in \eqref{eq:partial_exp_case3_2} .

For the left hand side of \eqref{eq:partial_exp_case3_1}, we have
\begin{equation*}
\begin{aligned}
  &\sum_{n=n_k}^{n_{k+1}-1} \left( m(u_{H,1}^{n+1} -u_{H,1}^{n}) + u_{H,2}^{n_{k}} -u_{H,2}^{n_{k-1}}, u_{H,1}^{n+1} -u_{H,1}^{n} \right)  \\
  & =  \sum_{n=n_k}^{n_{k+1}-1}  m \|  u_{H,1}^{n+1} -u_{H,1}^{n}\|^2 + \left(u_{H,2}^{n_{k}} -u_{H,2}^{n_{k-1}},  u_{H,1}^{n_{k+1}} -u_{H,1}^{n_k} \right) \\
 & \geq  \frac{1}{2}\|  u_{H,1}^{n_{k+1}} -u_{H,1}^{n_k}\|^2 -\frac{\gamma^2}{2}  \|u_{H,2}^{n_{k}} -u_{H,2}^{n_{k-1}}\|^2
\end{aligned}
\end{equation*}
For the left hand side of  \eqref{eq:partial_exp_case3_2}, we have
\begin{equation*}
\left(  u_{H,2}^{n_{k+1}} -u_{H,2}^{n_k} +u_{H,1}^{n_k} -u_{H,1}^{n_{k-1}} ,   u_{H,2}^{n_{k+1}} -u_{H,2}^{n_k} \right) \geq   \frac{1}{2}\|  u_{H,2}^{n_{k+1}} -u_{H,2}^{n_k}\|^2 -\frac{\gamma^2}{2}  \|u_{H,1}^{n_{k}} -u_{H,1}^{n_{k-1}}\|^2
\end{equation*}

Sum up the right hand side of \eqref{eq:partial_exp_case3_1} and \eqref{eq:partial_exp_case3_2}, we have
% \begin{equation*}
% \begin{aligned}
%      & - \Delta T  \sum_{n=n_k}^{n_{k+1}-1} a( \rev{u_{H,1}^{n+1}}, u_{H,1}^{n+1} -u_{H,1}^{n})   + (1-\omega) \Delta T a(u_{H,1}^{n_{k+1}} - u_{H,1}^{n_k} , u_{H,2}^{n_{k+1}} -u_{H,2}^{n_{k}} )  \\
%     & - \Delta T \left(  a(u_{H,2}^{n_k}, u_{H,1}^{n_{k+1}} -u_{H,1}^{n_k})- a(u_{H,1}^{n_{k+1}} , u_{H,2}^{n_{k+1}} -u_{H,2}^{n_{k}} ) \rev{+}a( u_{H,2}^{n_k}, u_{H,2}^{n_{k+1}} -u_{H,2}^{n_k})   \right) \\
%     \leq & -\frac{ \Delta T}{2}  \sum_{n=n_k}^{n_{k+1}-1}\|u_{H,1}^{n+1} -u_{H,1}^{n}\|^2_a - \frac{\Delta T}{2} \| u_H^{n_{k+1}}\|_a^2+ \frac{\Delta T}{2} \|  u_H^{n_{k}}\|_a^2\\
%     &+    \frac{(1-\omega)\Delta T}{2m} \| u_{H,1}^{n_{k+1}} -u_{H,1}^{n_k}\|_a^2 + \frac{(m-m\omega+1)\Delta T}{2} \|u_{H,2}^{n_{k+1}} -u_{H,2}^{n_{k}} \|_a^2\\
%     \leq &  -\frac{\omega \Delta T }{2}  \sum_{n=n_k}^{n_{k+1}-1}\|u_{H,1}^{n+1} -u_{H,1}^{n}\|^2_a - \frac{\Delta T}{2} \| u_H^{n_{k+1}}\|_a^2 \\
%     &+ \frac{\Delta T}{2} \|  u_H^{n_{k}}\|_a^2 + \frac{(m-m\omega+1)\Delta T}{2} \|u_{H,2}^{n_{k+1}} -u_{H,2}^{n_{k}} \|_a^2 \\
%     \leq &- \frac{\Delta T}{2} \| u_H^{n_{k+1}}\|_a^2+ \frac{\Delta T}{2} \|  u_H^{n_{k}}\|_a^2 + \frac{(m-m\omega+1)\Delta T}{2} \|u_{H,2}^{n_{k+1}} -u_{H,2}^{n_{k}} \|_a^2.
% \end{aligned}
% \end{equation*}

\begin{equation*}
\begin{aligned}
     & - \Delta T  \sum_{n=n_k}^{n_{k+1}-1} a( \rev{u_{H,1}^{n+1}}, u_{H,1}^{n+1} -u_{H,1}^{n})   + (1-\omega) \Delta T a(u_{H,1}^{n_{k+1}} - u_{H,1}^{n_k} , u_{H,2}^{n_{k+1}} -u_{H,2}^{n_{k}} )  \\
    & - \Delta T \left(  a(u_{H,2}^{n_k}, u_{H,1}^{n_{k+1}} -u_{H,1}^{n_k})\rev{+} a(u_{H,1}^{n_{k+1}} , u_{H,2}^{n_{k+1}} -u_{H,2}^{n_{k}} ) \rev{+}a( u_{H,2}^{n_k}, u_{H,2}^{n_{k+1}} -u_{H,2}^{n_k})   \right) \\
         \leq & \rev{-\frac{ \Delta T}{2}  \sum_{n=n_k}^{n_{k+1}-1}\|u_{H,1}^{n+1} -u_{H,1}^{n}\|^2_a - \frac{\Delta T}{2} \| u_{H,1}^{n_{k+1}}\|_a^2+ \frac{\Delta T}{2} \|  u_{H,1}^{n_{k}}\|_a^2}\\
        & \rev{-\Delta T a(u_{H,2}^{n_k}, u_{H,1}^{n_k})+ \Delta T a(u_{H,1}^{n_{k+1}}, u_{H,2}^{n_{k+1}})}\\
    & \rev{+    \frac{(1-\omega)\Delta T}{2m} \| u_{H,1}^{n_{k+1}} -u_{H,1}^{n_k}\|_a^2 + \frac{m(1-\omega)\Delta T}{2} \|u_{H,2}^{n_{k+1}} -u_{H,2}^{n_{k}} \|_a^2}\\ 
     & \rev{+    \frac{\Delta T}{2} \left( \| u_{H,2}^{n_k}\|_a^2 + \|u_{H,2}^{n_{k+1}} -u_{H,2}^{n_{k}}\|_a^2 - \| u_{H,2}^{n_{k+1}}\|_a^2 \right)}\\ 
     \leq & \rev{ -  \frac{\omega\Delta T}{2m} \| u_{H,1}^{n_{k+1}} -u_{H,1}^{n_k}\|_a^2  - \frac{\Delta T}{2} \| u_H^{n_{k+1}}\|_a^2+ \frac{\Delta T}{2} \|  u_H^{n_{k}}\|_a^2}\\
     &\rev{+ \frac{(m-m\omega+1)\Delta T}{2} \|u_{H,2}^{n_{k+1}} -u_{H,2}^{n_{k}} \|_a^2}\\
    \leq &- \frac{\Delta T}{2} \| u_H^{n_{k+1}}\|_a^2+ \frac{\Delta T}{2} \|  u_H^{n_{k}}\|_a^2 + \frac{(m-m\omega+1)\Delta T}{2} \|u_{H,2}^{n_{k+1}} -u_{H,2}^{n_{k}} \|_a^2. \\
  \end{aligned}
\end{equation*}

Combine the results, we have 
\begin{equation*}
\begin{aligned}
&\frac{\gamma^2}{2}\sum_{i=1}^2 \|  u_{H,i}^{n_{k+1}} -u_{H,i}^{n_k}\|^2 + \frac{1-\gamma^2}{2}\sum_{i=1}^2 \|  u_{H,i}^{n_{k+1}} -u_{H,i}^{n_k}\|^2  +\frac{  \Delta T}{2}\|u_{H}^{n_{k+1}}\|_a^2 \\
    &\leq   \frac{\gamma^2}{2} \sum_{i=1}^2 \|u_{H,i}^{n_{k}} -u_{H,i}^{n_{k-1}}\|^2 + \frac{  \Delta T}{2}  \|u_{H}^{n_k}\|_a^2 + \frac{(m-m\omega+1)\Delta T}{2} \|u_{H,2}^{n_{k+1}} -u_{H,2}^{n_{k}} \|_a^2
\end{aligned}
\end{equation*}
The stability estimate  is then obtained by using the stability condition \eqref{eq:stablility_case3}.
\end{proof}

% We remark that if $\omega=1$, the stability condition reduces to $\displaystyle{\Delta T  \sup_{v\in V_{H,2}} \frac{\|v\|_a^2}{\|v\|^2} \leq (1-\gamma^2)}$. This indicates we need the coarse time step size (used for the explicit part) satisfying the condition.  

To this end, we formulate the main theorem as follows.

\begin{theorem}
    Let $0=T_0< T_1< \cdots<  T_{(N-1)/m} = T$ be a coarse partition of the time domain $(0,T]$, and $\Delta T$ be the coarse time step size. Using the multirate time stepping in all coarse block $(T_k, T_{k+1}]$ ($k=0, \cdots, (N-1)/m-1$), we will obtain a stable scheme if 
    \begin{equation*}
           \Delta T  \sup_{v\in V_{H,2}} \frac{\|v\|_a^2}{\|v\|^2} \leq (1-\gamma^2)
    \end{equation*}
    for $\omega = 1$. 
\end{theorem}
This result can be easily obtained by Lemma \ref{lemma1} and Lemma \ref{lemma2}.

\textit{Remark}: We know that the time step size of explicit methods for the parabolic equation  scales as $H^2/\max(\kappa)$, where $\kappa$ is the diffusion coefficient. With the construction of basis for $V_{H,2}$ in section \ref{sec:V2}, we can demonstrate that the term $\sup_{v\in V_{H,2}} \frac{\|v\|_a^2}{\|v\|^2}$ in the stability condition is contrast independent. To illustrate the idea, we consider a simplified case, let $K_i$ be a square coarse element with only one vertical channel in the middle of the block. Let $K_{i,m}^1$ be the region on the left of the channel $K_{i,f}$, and $K_{i,m}^2$ be the region on the right of the channel $K_{i,f}$. Let $(x_1, y_1)$ and $(x_2, y_2)$ be the coordinates of the bottom-left and top-right vertices in $K_{i,m}^2$. Define a bubble function $B_i(x,y) \in C_{0}^{\infty} (K_i)$ such that 
\begin{align*}
    B_i(x,y) &= 0 \text{ in } K_{i,f} \cup K_{i,m}^1, \\
    B_i(x,y) &= \frac{64}{H^4}(x-x_1)(x_2-x)(y-y_1)(y_2-y).
\end{align*}
Then we have $\|B_i\|_{L^\infty} = 1$, and $\|\nabla B_i\|_{L^\infty} \leq CH^{-1}$. 

We now show that $\|w\|_{a(K_i)} \leq D H^{-1} \|w\|_{L^2(K_i)}$ for $w \in V_{H,2}(K_i)$. Take $w \in V_{H,2}(K_i)$, we have 
\begin{align*}
       a_i(w, v) + \mu^{(i)}_1 \int_{K_{i,f}} v  + \mu^{(i)}_2 \int_{K_{i,m}} v &=0, \quad \forall v \in V_0(K_i) \\
    \int_{K_{i,f}} w= 0, \quad \int_{K_{i,m}} w   &=  \pi_i(w),\\
\end{align*}
where $\pi_i$ is the projection from $L^2(K_i)$ to $V_{\text{aux},2}(K_i)$. Thus we have 
\begin{equation*}
    a_i(w,w) = -\mu^{(i)}_2 \int_{K_{i,m}} w \leq C_1 \|\mu^{(i)}_2\|_{L^2(K_i)} \|w\|_{L^2(K_i)}.
\end{equation*}
On the other hand, let $v=B_i \mu^{(i)}_2$ with $B_i$ defined above, we have 
\begin{align*}
     \|\mu^{(i)}_2\|_{L^2(K_i)}^2  &\leq C_2 \int_{K_{i,m}} B_i (\mu^{(i)}_2)^2\\
     &= -C_2 a_i(w, B_i \mu^{(i)}_2) \\
     & \leq  C_3 \|w\|_{a(K_i)} \|B_i \mu^{(i)}_2\|_{a(K_i)}\\
     &\leq D H^{-1}\|w\|_{a(K_i)} \|\mu^{(i)}_2\|_{L^2(K_i)}.
\end{align*}
Combine the results, we have $\|w\|_{a(K_i)} \leq D H^{-1}  \|w\|_{L^2(K_i)}$.

We remark that this idea can be extended to more general $K_i$ with one smooth channel in it by appropriate coordinate transformation.

As for the constant $\gamma$, we can observe that it is strictly less than 1. Let $v_1 \in V_{H,1}$ and $v_2 \in V_{H,2}$, and $P: L^2 \rightarrow V_{\text{aux},1}+V_{\text{aux},2}$ be a projection operator such that $P (v) = \sum_{i} \frac{1}{|K_{i,m}|}\int_{K_{i,m}} v + \sum_{j} \frac{1}{|K_{j,f}|}\int_{K_{j,f}} v$, then we have
	\begin{equation*}
	    (v_1, v_2) = (P v_1, P v_2) + ((I-P)v_1, (I-P)v_2) = ((I-P)v_1, (I-P)v_2),
	\end{equation*}
	since $(P v_1, P v_2) = 0$.
	Thus, 
	\begin{align*}
	    \frac{(v_1, v_2) }{\|v_1\|_{L^2} \|v_2\|_{L^2}}=\frac{((I-P)v_1, (I-P)v_2)}{\|v_1\|_{L^2} \|v_2\|_{L^2}} 
	    \leq \frac{\|(I-P)v_1\|_{L^2} }{\|v_1\|_{L^2} }\frac{\|(I-P)v_2\|_{L^2} }{\|v_2\|_{L^2} } < 1
	\end{align*}
	since $\|(I-P)v_i\|_{L^2}^2 <\|(I-P)v_i\|_{L^2}^2+\|P v_i\|_{L^2}^2 = \|v_i\|_{L^2}^2$.
	Actually, since $v_i$ solves the constraint minimizing problem with the energy $a(v_i,v_i)$ minimized, $v_1$ and $v_2$ then minimizes the oscillation in the fractured region and matrix region, respectively. On the other hand, $P v_i$ is the piecewise constant representing the average of $v_i$ in the corresponding region. Thus, $\|(I-P)v_i\|_{L^2}$ is relatively small compared with $\|P v_i\|_{L^2}$, and $\frac{\|(I-P)v_i\|_{L^2} }{\|v_i\|_{L^2}}$ should be away from $1$. %So we have $\gamma  < 1$.

\rev{We remark that our proposed method provides an adaptive time refinement strategy which focuses on the case when the temporal error is large (for example, at the short-time simulation period or when the source term is changing). In the long term simulation, the spatial error will dominate no matter which time discritization scheme is employed. To satisfy the stability condition, the time step size in our partially explicit scheme should be suitably coupled to the mesh parameter. Though the purely implicit scheme without splitting is unconditionally stable, the proposed multirate approach still has some advantages: (1) If the temporal error is large, a smaller time step size is needed to reduce the error. In this case, the partially explicit scheme is computational faster compared with the purely implicit scheme. (2) If the temporal error is relatively small compared with the spatial error, our error indicators will decide not to refine the time step size. As long as the coarse time step size satisfy the stability condition, the partially explicit scheme is preferable in terms of computational efficiency. (3) If a very large time step is employed, one can still use the proposed splitting method with implicit discretization in both equations. In this case, the stability can be proved in a similar fashion, and the multirate method we designed can still be employed to reduce the error. }

\subsection{Adaptive multirate algorithm based on the residual} \label{sec:main_method}
In this section, we will propose a new adaptive multirate algorithm to select a suitable time step size for the implicit-explicit scheme. The idea is to derive an error indicator based on residuals, the indicators will give an estimate of the errors if we use coarse time discretization for both equations in implicit-explicit scheme \eqref{eq:partial_exp1}-\eqref{eq:partial_exp2}. Then one can adaptively refine the time step size for the part whose error is large. We first show the derivation of the error estimators, and then present our main adaptive algorithm.
% \subsection{Residual and error indicators}
Let $U(t)$ be the piecewise linear function with $U(T^n) = u_H^n = u_{H,1}^n + u_{H,2}^n$ such that on $(T_n, T_{n+1}]$ 
\begin{equation*}
    U(t) = u_H^n + \frac{t-T_n}{\Delta T} (u_H^{n+1}-u_H^n),
\end{equation*}
and $F(t)$ be the piecewise constant such that on each time interval $(T_n, T_{n+1}]$, it is equal to the $L^2$-projection of $f^{n+1}$ onto the multiscale space $V_H$, i.e.
\[(F(t), v) = (f^{n+1}, v)\]
for all $v\in V$.

We introduce the space 
\begin{equation*}
    X_n = L^2((T_n, T_{n+1}];H^1(\Omega)) %\cap H^1((T_n, T_{n+1}];L^2(\Omega))
\end{equation*}
and define 
\begin{equation*}
    (u,v)_{X_n} =  \int_{T_n}^{T_{n+1}} \left(\int_{\Omega} u v + \int_{\Omega} \kappa \nabla u \cdot  \nabla v  \right) dt, 
\end{equation*}
\begin{equation*}
    \|v\|_{X_n} = \left( \int_{T_n}^{T_{n+1}} \left( \|v\|^2 + \| v\|^2_a  \right) dt  \right)^{\frac{1}{2}}, 
\end{equation*}
where $\|v\|^2_a  = \|\kappa \nabla v\|^2$. \rev{We remark that the space $X_n$ is a common choice of space for parabolic problems, and $\|v\|_{X_n}$ is the associated norm \cite{vohralik2013posteriori}. }

%+ \|v_t\|^2_{L^2(\Omega)}
Define a constant $\gamma_x$ depending on $V_{H,1}$ and $V_{H,2}$ as
\rev{
  \begin{equation}\label{eq:gammax}
  \gamma_x = \sup_{\substack{ v_1 \in L^2((T_n, T_{n+1}];V_{H,1}),\\ v_2 \in L^2((T_n, T_{n+1}];V_{H,2}) }} \frac{(v_1, v_2)_{X_n}}{||v_1||_{X_n} ||v_2||_{X_n} } < 1.
    %   \gamma_x = \sup_{v_1 \in V_{H,1}, v_2 \in V_{H,2}} \frac{(v_1, v_2)_{X_n}}{||v_1||_{X_n} ||v_2||_{X_n} } < 1.
  \end{equation} }

%and denote by $$\displaystyle{\|v\|_{V} =\|v\|_{L^2(\Omega)} + \|\kappa \nabla v\|_{L^2(\Omega)}  }.$$

Let the \rev{integral of the residual over the time interval $(T_n, T_{n+1}]$} be
\rev{
\begin{equation*}
  \int_{T_n}^{T_{n+1}} (R(U(t)), v) dt = \int_{T_n}^{T_{n+1}} (f(t), v)dt - \int_{T_n}^{T_{n+1}} (U'(t), v) dt- \int_{T_n}^{T_{n+1}} a(U(t), v)dt,
\end{equation*}
}
% where $v(\cdot, t) \in V_H$, and $v=v_1 + v_2$ where $v_1(\cdot, t) \in V_{H,1}$ and $v_2(\cdot, t) \in V_{H,2}$.
% Let 
\rev{and define two dual norms of the residual 
\begin{align*}
    \mathcal{R}_v &= \sup_{\substack{ v \in L^2((T_n, T_{n+1}];V_{H}),\\ \|v\|_{X_n} = 1}} \int_{T_n}^{T_{n+1}}(R(U(t)), v) dt,\\
    \mathcal{R}_x &= \sup_{v \in X_n, \|v\|_{X_n} = 1} \int_{T_n}^{T_{n+1}}(R(U(t)), v) dt,
\end{align*}}
then we have the following estimates.

\begin{theorem}
    Define the following error indicators\\
    \textbf{Type 1}:
    \begin{equation}\label{eq:type1}
    \begin{aligned}
    \eta_1^n = \sqrt{\frac{\Delta T }{3}} (\| u_{H,1}^{n+1} -u_{H,1}^n\|_a + \rev{\gamma}\| u_{H,2}^{n+1} -u_{H,2}^n\|_a) +  \Delta T^{\frac{3}{2}} \|\partial^2_t  u_{H,2}^{n+1} \|,\\
           \eta_2^n = \rev{\gamma} \sqrt{C_{\omega}\Delta T }\| u_{H,1}^{n+1} -u_{H,1}^n\|_a + \frac{\sqrt{\Delta T}}{3}  \| u_{H,2}^{n+1} -u_{H,2}^n\|_a  +  \Delta T^{\frac{3}{2}} \|\partial^2_t  u_{H,1}^{n+1} \|.
    \end{aligned}
    \end{equation}
            \textbf{Type 2}:
        \begin{equation}\label{eq:type2}
    \begin{aligned}
        \eta_1^n = \sqrt{\frac{\Delta T }{3}}  ( \| u_{H,1}^{n+1} -u_{H,1}^n\|_a +\rev{\gamma} \| u_{H,2}^{n+1} -u_{H,2}^n\|_a) +  \Delta T^{\frac{3}{2}} \|\partial^2_t  u_{H,2}^{n+1} \|_{a^*},\\
           \eta_2^n = \rev{\gamma} \sqrt{C_{\omega}\Delta T }\| u_{H,1}^{n+1} -u_{H,1}^n\|_a +  \frac{\sqrt{\Delta T}}{3}  \| u_{H,2}^{n+1} -u_{H,2}^n\|_a  +  \Delta T^{\frac{3}{2}} \|\partial^2_t  u_{H,1}^{n+1} \|_{a^*}.
        \end{aligned}
      \end{equation}
      where $\displaystyle{C_{\omega}=(\frac{1}{3} + \omega^2- \omega)}$, and \rev{$\gamma$ is the constant defined in \eqref{eq:gamma}}.
      
      Then there exists constant $D_1, D_2$ such that 
      \begin{equation}
       \rev{\mathcal{R}_v} \leq D_1 (1-\gamma_x)^{-\frac{1}{2}} \left((\eta_1^n)^2 +(\eta_2^n)^2  \right)^{\frac{1}{2}}  + D_2 \|f(t)-F(t)\|_{L^2((T_n,T_{n+1}], L^2(\Omega))}
        %  \mathcal{R} \leq C_1  \|f-F(t)\|_{L^2((0,T], L^2(\Omega))} + C_2\sum_n \eta_1^n +  C_3 \sum_n \eta_2^n
      \end{equation}
\end{theorem}

\begin{proof}
By definition, we have
  \begin{equation*}
\begin{aligned}
  &\rev{ \int_{T_n}^{T_{n+1}}} (R(U(t)), v)  dt = \rev{\int_{T_n}^{T_{n+1}} } \bigg[ (f(t), v) -(f^{n+1}, v) - (U'(t), v) -a(U(t), v)+  \\
  & \left( \frac{u_{H,1}^{n+1} -u_{H,1}^n }{\Delta T}, v_1 \right) + \left( \frac{u_{H,2}^{n} -u_{H,2}^{n-1} }{\Delta T}, v_1 \right)  + a(u_{H,1}^{n+1} +u_{H,2}^{n}, v_1)+ \\
  &\left( \frac{u_{H,2}^{n+1} -u_{H,2}^n }{\Delta T}, v_2 \right) +\left( \frac{u_{H,1}^{n} -u_{H,1}^{n-1} }{\Delta T}, v_2 \right) + a((1-\omega)u_{H,1}^{n} +\omega  u_{H,1}^{n+1} +u_{H,2}^{n}, v_2)  \bigg] dt.
  \end{aligned}
  \end{equation*}
 By the definition of $U(t)$, we have $U'(t) = \frac{u_H^{n+1}-u_H^n}{\Delta T}$, then
 \begin{equation*}
 \begin{aligned}
        & (U'(t), v) -\left( \frac{u_{H,1}^{n+1} -u_{H,1}^n }{\Delta T}, v_1 \right) - \left( \frac{u_{H,2}^{n+1} -u_{H,2}^n }{\Delta T}, v_2 \right)\\
     &=\left( \frac{u_{H,1}^{n+1} -u_{H,1}^n }{\Delta T}, v_2 \right)  +\left( \frac{u_{H,2}^{n+1} -u_{H,2}^n }{\Delta T}, v_1\right).
 \end{aligned}
 \end{equation*}
 Further, we have
 \begin{equation*}
     \begin{aligned}
  & u_{H,1}^{n+1} +u_{H,2}^{n}-U(t) = \frac{T_{n+1}-t}{\Delta T} (  u_{H,1}^{n+1} -u_{H,1}^{n}) - \frac{t-T_{n}}{\Delta T} (  u_{H,2}^{n+1} -u_{H,2}^{n}),\\
   &(1-\omega)u_{H,1}^{n} +\omega  u_{H,1}^{n+1} +u_{H,2}^{n}-U(t)   \\
  = & (\omega-\frac{T_{n+1}-t}{\Delta T}) (  u_{H,1}^{n+1} -u_{H,1}^{n})- \frac{t-T_{n}}{\Delta T} (  u_{H,2}^{n+1} -u_{H,2}^{n}).
  \end{aligned}  
 \end{equation*}
 Thus, we can write
\begin{equation*}
\rev{ \int_{T_n}^{T_{n+1}}} (R(U(t)), v)) dt = \rev{ \int_{T_n}^{T_{n+1}}} \left[ (f(t), v) -(f^{n+1}, v)  + (R_1,v_1) +(R_2,v_2)\right] dt \end{equation*}
where
\begin{equation*}
  \begin{aligned}
    \rev{ \int_{T_n}^{T_{n+1}}} (R_1,v_1)dt &= \rev{\int_{T_n}^{T_{n+1}}} [ \frac{\rev{T}_{n+1} - t}{\Delta T } a(u_{H,1}^{n+1} -u_{H,1}^{n}, v_1)  - \frac{t-\rev{T_n}}{\Delta T} a(u_{H,2}^{n+1} -u_{H,2}^{n}, v_1) ] dt\\
    & -\rev{ \int_{T_n}^{T_{n+1}}} \Delta T  \left( \frac{u_{H,2}^{n+1} -2u_{H,2}^n +u_{H,2}^{n-1} }{\Delta T^2}, v_1\right)dt \\
     \rev{ \int_{T_n}^{T_{n+1}}} (R_2,v_2)dt &= \rev{ \int_{T_n}^{T_{n+1}}} \big( \omega- \frac{\rev{T}_{n+1} - t}{\Delta T} \big) a(u_{H,1}^{n+1} -u_{H,1}^{n}, v_2)dt\\
     &-  \rev{ \int_{T_n}^{T_{n+1}}} \frac{t-\rev{T}_n}{\Delta T } a(u_{H,2}^{n+1} -u_{H,2}^{n}, v_2) dt\\
     & -\rev{ \int_{T_n}^{T_{n+1}}}\Delta T \left( \frac{u_{H,1}^{n+1} -2u_{H,1}^n + u_{H,1}^{n-1}}{\Delta T^2}, v_2 \right)dt.
\end{aligned}  
\end{equation*}
Integrate from $T_n$ to $T_{n+1}$, we get
%    \begin{equation}\label{eq:eta1}
    % \eta_1^n = \frac{\sqrt{\tau }}{3} (\| u_{H,1}^{n+1} -u_{H,1}^n\|_a + \| u_{H,2}^{n+1} -u_{H,2}^n\|_a) +  \tau^{\frac{3}{2}} \|\partial^2_t  u_{H,2}^{n+1} \|,
    %     \end{equation}
    %     \begin{equation}\label{eq:eta2}
    %   \eta_2^n = (\frac{1}{3} + \omega^2- \omega)\sqrt{\tau }\| u_{H,1}^{n+1} -u_{H,1}^n\|_a + \frac{\sqrt{\tau}}{3}  \| u_{H,2}^{n+1} -u_{H,2}^n\|_a  +  \tau^{\frac{3}{2}} \|\partial^2_t  u_{H,1}^{n+1} \|.
    %   \end{equation}

\begin{equation}\label{eq:ineq_R1}
\begin{aligned}
     & \int_{T_n}^{T_{n+1}} (R_1, v_1) dt  \leq C_1 \sqrt{\frac{\Delta T }{3}} \| u_{H,1}^{n+1} -u_{H,1}^n\|_a \left( \int_{T_n}^{T_{n+1}} \|v_1\|_a^2 dt\right)^{\frac{1}{2}} \\
     &\ + C_2 \rev{\gamma} \sqrt{\frac{\Delta T }{3}}  \| u_{H,2}^{n+1} -u_{H,2}^n\|_a \left( \int_{T_n}^{T_{n+1}} \|v_1\|_a^2 dt\right)^{\frac{1}{2}}   + C_3 \Delta T^{\frac{3}{2}}  E(\partial^2_t  u_{H,2}^{n+1}, v_1),
 \end{aligned}
 \end{equation}
 \begin{equation}\label{eq:ineq_R2}
\begin{aligned}
   & \int_{T_n}^{T_{n+1}} (R_2, v_2)dt   \leq C_1' \rev{\gamma}  \sqrt{(\frac{1}{3} + \omega^2- \omega) \Delta T} \| u_{H,1}^{n+1} - u_{H,1}^n\|_a \left( \int_{T_n}^{T_{n+1}} \|v_2\|_a^2 dt \right)^{\frac{1}{2}} \\
     & + C_2' \sqrt{\frac{\Delta T }{3}}  (\| u_{H,2}^{n+1} -u_{H,2}^n\|_a \left( \int_{T_n}^{T_{n+1}} \|v_2\|_a^2 dt\right)^{\frac{1}{2}}   + C_3' \Delta T^{\frac{3}{2}}  E(\partial^2_t  u_{H,1}^{n+1}, v_2),
\end{aligned}
    \end{equation}
where \rev{$\gamma$ is }defined in \eqref{eq:gamma}.

In the above inequalities, $\partial^2_t(u_{H,i}^{n+1}) = \frac{u_{H,i}^{n+1} - 2u_{H,i}^{n}+u_{H,i}^{n-1}}{\Delta T^2}$, which stands for the approximation of second derivative with respect to time, and we have
\begin{equation*}
  \begin{aligned}
E(\partial^2_t  u_{H,2}^{n+1}, v_1) &= \|\partial^2_t  u_{H,2}^{n+1} \|^2    \left( \int_{T_n}^{T_{n+1}} \|v_1\|^2 dt\right)^{\frac{1}{2}},\\
 E(\partial^2_t  u_{H,1}^{n+1}, v_2)&=   \|\partial^2_t  u_{H,1}^{n+1} \|^2    \left( \int_{T_n}^{T_{n+1}} \|v_2\|^2 dt\right)^{\frac{1}{2}}
\end{aligned}  
\end{equation*}
to derive the first type of indicators $\eta_1$, $\eta_2$ as defined in \eqref{eq:type1}, or 
\begin{equation*}
  \begin{aligned}
E(\partial^2_t  u_{H,2}^{n+1}, v_1) &= \|\partial^2_t  u_{H,2}^{n+1} \|_{a^*}   \left( \int_{T_n}^{T_{n+1}} \|v_1\|_a^2 dt\right)^{\frac{1}{2}},\\
 E(\partial^2_t  u_{H,1}^{n+1}, v_2)&=   \|\partial^2_t  u_{H,1}^{n+1} \|_{a^*}   \left( \int_{T_n}^{T_{n+1}} \|v_2\|_a^2 dt\right)^{\frac{1}{2}}
\end{aligned}  
\end{equation*}
to derive the second type of indicators $\eta_1$, $\eta_2$ in \eqref{eq:type2}. Here $||\cdot||_{a^*}$ is the dual norm.

Add the two inequalities in \eqref{eq:ineq_R1} and \eqref{eq:ineq_R2} together and by the definition of $X_n$ norm, for both types of indicators, we have
\begin{equation*}
\begin{aligned}
     \int_{T_n}^{T_{n+1}} (R_1, v_1) dt+ \int_{T_n}^{T_{n+1}} (R_2, v_2)dt &\leq  C\eta_1^n \|v_1\|_{X_n} + C' \eta_2^n \|v_2\|_{X_n}   \\
    &\leq D_1 \left( (\eta_1^n)^2 +(\eta_2^n)^2  \right)^{\frac{1}{2}}\left(  \|v_1\|_{X_n}^2 +\|v_2\|_{X_n}^2  \right)^{\frac{1}{2}} \\
    &\leq D_1 (1-\gamma_x)^{-\frac{1}{2}} \left((\eta_1^n)^2 +(\eta_2^n)^2  \right)^{\frac{1}{2}}    \|v\|_{X_n},
\end{aligned}
\end{equation*}
where $\eta_1^n$ and $\eta_2^n$ are defined in \eqref{eq:type1} or \eqref{eq:type2}. In the last step of the above derivation, we use the fact that
\begin{equation*}
    (v_1, v_2)_{X_n} \leq \gamma_x \|v_1\|_{X_n} \|v_2\|_{X_n}
\end{equation*}
by the definition \eqref{eq:gammax}, which indicates
\begin{equation*}
\begin{aligned}
    \|v_1+v_2\|_{X_n}^2 &\geq \|v_1\|_{X_n}^2 + \|v_2\|_{X_n}^2 -2\gamma_x \|v_1\|_{X_n} \|v_2\|_{X_n}\\
    & \geq (1-\gamma_x) \left( \|v_1\|_{X_n}^2 + \|v_2\|_{X_n}^2  \right).
\end{aligned}
\end{equation*}
% + \int_{T_n}^{T_{n+1}} (f(t)-F(t), v)dt
%

Finally, take the sup with respect to $v\in X_n$, we have the following estimate
\begin{equation*}
\begin{aligned}
   \rev{\mathcal{R}_x} &  \leq D_1 (1-\gamma_x)^{-\frac{1}{2}} \left((\eta_1^n)^2 +(\eta_2^n)^2  \right)^{\frac{1}{2}}  + D_2 \|f(t)-F(t)\|_{L^2((T_n,T_{n+1}], L^2(\Omega))}.
\end{aligned}
\end{equation*}
 \rev{It is obvious that $\mathcal{R}_v\leq \mathcal{R}_x$. On the other hand, we observe that
 \begin{equation*}
     \int_{T_n}^{T_{n+1}} (R(U(t)), v)) dt =  \int_{T_n}^{T_{n+1}} (R(U(t)), v-\Pi v)) dt + \int_{T_n}^{T_{n+1}} (R(U(t)), \Pi v)) dt,
 \end{equation*}
 where $\Pi: H_0^1(\Omega) \rightarrow V_H$ is a projection operator. Thus 
 \begin{align*}
     &\sup_{v \in X_n} \frac{\int_{T_n}^{T_{n+1}}(R(U(t)), v) dt}{\|v\|_{X_n}} \\
     \leq &\sup_{v \in L^2((T_n, T_{n+1}];V_{H}^{\perp})} \frac{\int_{T_n}^{T_{n+1}}(R(U(t)), v) dt }{\|v\|_{X_n}} + \sup_{v \in L^2((T_n, T_{n+1}];V_{H})} \frac{\int_{T_n}^{T_{n+1}}(R(U(t)), v) dt }{\|v\|_{X_n}}, 
 \end{align*}
 this implies 
 \begin{equation*}
      \mathcal{R}_x \leq \sup_{v \in L^2((T_n, T_{n+1}];V_{H}^{\perp})} \frac{\int_{T_n}^{T_{n+1}}(R(U(t)), v) dt }{\|v\|_{X_n}} + \mathcal{R}_v.
 \end{equation*}
We assume that the space $V_{H}$ provides a good approximation to $V$, then the supremum term over $L^2((T_n, T_{n+1}];V_{H}^{\perp})$ will be relatively small compared to $\mathcal{R}_v$.}
 \end{proof}

% \begin{equation*}
%     \left( \int_{t_n}^{t_{n+1}} \left(\| R_1\|_{H^{-1}(\Omega)}^2 \right) dt \right)^{\frac{1}{2}} \leq 
% C\left( \frac{\sqrt{\tau }}{3} (\| u_{H,1}^{n+1} -u_{H,1}^n\|_a + \| u_{H,2}^{n+1} -u_{H,2}^n\|_a) +  \tau^{\frac{3}{2}} \|\partial^2_t  u_{H,2}^{n+1} \| \right)    
% \end{equation*}

% \begin{equation*}
%     \left( \int_{t_n}^{t_{n+1}} \left(\| R_2\|_{H^{-1}(\Omega)}^2 \right) dt \right)^{\frac{1}{2}}  \leq 
% C' \left(  (\frac{1}{3} + \omega^2- \omega)\sqrt{\tau }\| u_{H,1}^{n+1} -u_{H,1}^n\|_a + \frac{\sqrt{\tau}}{3}  \| u_{H,2}^{n+1} -u_{H,2}^n\|_a  +  \tau^{\frac{3}{2}} \|\partial^2_t  u_{H,1}^{n+1} \| \right)    
% \end{equation*}

\rev{\textit{Remark}: In this work, we considered the fractured/channelized media with high contrast, and we assume the proposed multiscale space $V_{H}$ is good enough to approximate solutions in space. Our aim is to handle the error in the time discretization effectively via a multirate approach. For more general model problems with highly heterogeneous permeabilities, we will consider enriching the spatial approximation by constructing additional multiscale basis in our future work.}

\rev{\textit{Remark}: We defined two types of error indicators, the difference between the two lies in the norm of the term $\partial^2_t  u_{H,1}^{n+1}$. Both types have advantages and disadvantages. For type 1,  the computation of the indicators is more straightforward, but we may need different scales of the threshold parameters for $\delta_1$ and $\delta_2$ in practice (this can be observed in the numerical examples in section \ref{sec:numerical}. On the other hand, for type 2, the threshold parameters for $\delta_1$ and $\delta_2$ can be chosen consistently, but the computation for the dual norm of $\partial^2_t  u_{H,i}^{n+1}$ is less straightforward.}

\subsection{Main algorithm} \label{sec:alg}

In this part, we present the adaptive multi-time-step algorithm. At the beginning of the procedure, we solve the problem at the coarse time resolution. The coarse time step size is set to guarantee the stability of the scheme. Then we conduct refinement for the part of the equations \eqref{eq:partial_exp1}-\eqref{eq:partial_exp2} to the fine time resolution according to the error indicators and user-defined thresholds. If needed, the refinement will be implemented inside the current coarse block, and the solutions at the newest coarse time instance will be replaced. Then the time grid will be set back to the coarse resolution for both equations and the solver will march forward. The procedure will be performed iteratively \rev{until} the simulation is done.

In the following, the fine time step size is $\Delta t$, the coarse time step size is $\Delta T$, and  $\Delta T =m \Delta t$. The total number of coarse time steps is $N$. Let $\text{dim}(V_{H,1}) = d_1$, $\text{dim}(V_{H,2}) = d_2$, $\text{dim}(V_h) = D$, and let $\Psi_1 \in \mathbb{R}^{D\times d_1}$ and $\Psi_2 \in \mathbb{R}^{D\times d_2}$ be the matrices whose columns are the bases of $V_{H,1}$, $V_{H,2}$, respectively.
Let $M_f$ and $A_f$ be the fine scale mass matrix and stiffness matrix,
define the following coarse scale matrices
\begin{equation*}
\begin{aligned}
    %CEM_M
    M_{H,1} &=  \Psi_1^T M_f \Psi_1, \quad A_{H,1} = \Psi_1^T A_f \Psi_1,\\
    %SCEM_M22, SCEM_A22
    M_{H,2} &=  \Psi_2^T M_f \Psi_2, \quad A_{H,2} = \Psi_2^T A_f \Psi_2,\\
    %SCEM_M12, SCEM_A12
    M_{H,12} &=  \Psi_1^T M_f \Psi_2, \quad A_{H,12} = \Psi_2^T A_f \Psi_2,\\
    %CEM_f = loc_basis'*f; SCEM_f2 = loc_basis_add'*f;
    F_{H,1}^n & =  \Psi_1^T f^n, \quad \quad \quad   F_{H,2}^n =  \Psi_2^T f^n  
\end{aligned}
\end{equation*}
Let $U^n_{1,H}$ and $U^n_{2,H}$ be the coarse scale solution at time step $n$. Then the matrix equations can be displayed as 

    \begin{equation}\label{eq:partial_exp_mat1}
    \left(M_{H,1} + \tau A_{H,1} \right)U^{k+1}_{H,1}  = 
     M_{H,1} U^{k}_{H,1} + M_{H,12} (U^{k-1}_{H,2} -U^{k}_{H,2}) - \tau A_{H,12} 
    U^{k}_{H,2} + \tau F_{H,1}^{k+1},
    \end{equation}
\begin{equation}\label{eq:partial_exp_mat2}
\begin{aligned}
   M_{H,2} U^{k+1}_{H,2}  &= \left(M_{H,2}- \tau A_{H,2} \right)U^{k}_{H,2} + M_{H,12}^T (U^{k-1}_{H,1} -U^{k}_{H,1}) \\
   &- (1-\omega) \tau A_{H,12}^T U^{k}_{H,1} - \omega \tau A_{H,12}^T U^{k+1}_{H,1} +  \tau F_{H,2}^{k+1}.   
\end{aligned}
\end{equation}

% \begin{enumerate}
%     \item[]\textbf{Case 1}: 
%      \begin{equation}\label{eq:partial_exp_case1_mat}
%     \begin{aligned}
%     &\left(M_{H,1} + \Delta T A_{H,1} \right)U^{k+1}_{H,1}  = 
%      M_{H,1} U^{k}_{H,1} + M_{H,12} (U^{k-1}_{H,2} -U^{k}_{H,2}) - \Delta T A_{H,12} 
%     U^{k}_{H,2} + \Delta T F_{H,1}^k, \\
%     & M_{H,2} U^{k+1}_{H,2}  = \left(M_{H,2}-\Delta T A_{H,2} \right)U^{k}_{H,2} + M_{H,12}^T (U^{k-1}_{H,1} -U^{k}_{H,1}) - (1-\omega)\Delta T A_{H,12}^T U^{k}_{H,1} \\
%     & \quad \quad \quad  \quad \quad \quad - \omega\Delta T A_{H,12}^T U^{k+1}_{H,1} + \Delta T F_{H,2}^k
%      \end{aligned}
%     \end{equation}
% \end{enumerate}

Our proposed method can be summarized in the algorithm \ref{alg:adaptive}. 

\begin{algorithm}
 \caption{Adaptive multirate algorithm for partially explicit temporal splitting scheme }
    \begin{algorithmic}[1]
    \Procedure{Adaptive multirate}{Thresholds $\delta_1$, $\delta_2$, Initial condition $u_0$} 
    \State Define matrices $\Psi_1 \in \mathbb{R}^{D\times d_1}$, $\Psi_1 \in \mathbb{R}^{D\times d_2}$ using multiscale basis in $V_{H,1}$, $V_{H,2}$
    
    \ForAll{$k = 1: m$}
        \State $\tau \gets \Delta t$
        \State  Solve equation \eqref{eq:standard_scheme}
     \EndFor
    
    \State $U_{H,1}^0 \gets \Psi_1^T u_0$, $U_{H,2}^0 \gets \Psi_2^T u_0$  
    \State $U_{H,1}^1 \gets \Psi_1^T u_H^m$, $U_{H,2}^1 \gets \Psi_2^T u_H^m$ 
    
    \ForAll{$k = 1: N-1$}
        \State $\tau \gets \Delta T$ 
        \State Solve equations \eqref{eq:partial_exp_mat1} - \eqref{eq:partial_exp_mat2} 
        \State Save $U_{H,1}^{k+1}$ and $U_{H,2}^{k+1}$
        \State Compute $\rev{\eta_1^k}$, $\rev{\eta_2^k}$ from \eqref{eq:type1} or \eqref{eq:type2}
     	\If{ $\rev{\eta_1^k} > \delta_1$ and $\rev{\eta_2^k} < \delta_2$ }
     	    \State $\tilde{U}_{H,1}^0 \gets U_{H,1}^k$
    		\State Replace $k$ with $j$, denote $\tilde{U}_{H,1}^j = U_{H,1}^j$ in \eqref{eq:partial_exp_mat1}
    		\ForAll{$j = 1: m$}
    		    \State Solve equation \eqref{eq:partial_exp_mat1} with $\tau = \Delta t$ 
    		\EndFor
    		\State Replace $U_{H,1}^{k+1}$ with $\tilde{U}_{H,1}^{m+1}$
    		
   		 \ElsIf{$\rev{\eta_1^k} < \delta_1$ and $\rev{\eta_2^k} > \delta_2$ }
   		 	\State $\tilde{U}_{H,2}^0 \gets U_{H,2}^k$
    		\State Replace $k$ with $j$, let $\tilde{U}_{H,j}^2 = U_{H,j}^2$ in \eqref{eq:partial_exp_mat2}
    		\ForAll{$j = 1: m$}
    		    \State Solve equation \eqref{eq:partial_exp_mat2} with $\tau = \Delta t$ 
    		\EndFor
    	    \State Replace $U_{H,2}^{k+1}$ with $\tilde{U}_{H,2}^{m+1}$
    	 \ElsIf{$\rev{\eta_1^k} > \delta_1$ and $\rev{\eta_2^k} > \delta_2$ }
    	    
    	    \State $\tilde{U}_{H,0}^1 \gets U_{H,k}^1$, $\tilde{U}_{H,0}^2 \gets U_{H,k}^2$
    		\State Replace $k$ with $j$, let $\tilde{U}_{H,j}^1 = U_{H,j}^1$ in \eqref{eq:partial_exp_mat1} and $\tilde{U}_{H,j}^2 = U_{H,j}^2$ in \eqref{eq:partial_exp_mat2}
    		\ForAll{$j = 1: m$}
    		    \State Solve equations \eqref{eq:partial_exp_mat1}-\eqref{eq:partial_exp_mat2} with $\tau = \Delta t$ 
    		\EndFor
    		\State Replace $U_{H,1}^{k+1}$ with $\tilde{U}_{H,1}^{m+1}$
    		\State Replace $U_{H,2}^{k+1}$ with $\tilde{U}_{H,2}^{m+1}$
    		
       	\EndIf 

    \EndFor
    \State $u^N_{H,1} \gets \Psi_1 U_{H,1}^N $, $u^N_{H,2} \gets \Psi_2 U_{H,2}^N $
    \State \textbf{return}  $u^N_{H} = u^N_{H,1} + u^N_{H,2}$
    \EndProcedure
    \end{algorithmic}    \label{alg:adaptive}
\end{algorithm}

\section{Numerical examples} \label{sec:numerical}
In this section, we will present some numerical tests and demonstrate the performance of the proposed algorithm.
Consider the parabolic equation on a unit square domain $\Omega = [0,1]\times [0,1]$. Let the coarse mesh size be $H=0.1$ and the fine mesh size be $h=0.01$. We apply zero Dirichlet boundary conditions and zero initial conditions in the following examples. The reference solutions are computed using an even finer time discretization with Crank-Nicolson scheme, and the spatial discretizations were as discussed in Section \ref{sec:spaces}. \rev{We expect that our method has the spatial convergence rate $O(H)$ in the energy norm since the space $V_H = V_{H,1} + V_{H,2}$ is originated from NLMC \cite{cem-gmsfem, NLMC, zhao2020analysis}, and the temporal convergence rate is $O(\Delta T)$. Please refer to appendix \ref{appendix:thm} for a sketch of proof for spatial error.}

\subsection{Example 1: time-independent smooth source term}
In the first example, we use a smooth source term $f(x,y) = 2\pi^2 \sin(\pi x)\sin(\pi y) \exp(-(x-0.5)^2 -(y-0.5)^2 )$. The configuration of the permeability field can be found in Figure \ref{fig:source_perm_ex1}. The value of permeability is $10^4$ in the channel, and $1$ in the background.

The total simulation time $T=0.05$. The coarse time step size is $\Delta T = 10^{-4}$ and the fine time step size is $\Delta t = 10^{-5}$. We use the Crank-Nicolson scheme with $\delta t = 10^{-6}$ to compute reference solutions. The comparison of solutions at the different time steps are presented in Figure \ref{fig:sol_ex1}, where we have reference solutions on the left column, and the solutions on the right column are obtained from our proposed method to adaptively refine temporal mesh based on residuals. 

%We remark that, in this example, the results are similar using two types of error indicators, and we will only report those for the first type of indicators. 
The errors (evaluated at coarse time instances) are shown in Figure \ref{fig:errs_ex1_type1} for type 1 when we take $\delta_1 = 1.5\times 10^{-4}$, $\delta_2 = 5\times 10^{-6}$, and we compare the results using uniform fine time discretization, using uniform coarse time discretization and using adaptive time refinement discretization, correspondingly. The refinement indicators in the right of Figure \ref{fig:errs_ex1_type1} demonstrate that the algorithm automatically chooses different time step for the two equations \eqref{eq:partial_exp1} and \eqref{eq:partial_exp2}. %Here, we have $\delta_1=1.5\times 10^{-4}$, $\delta_2=1\times 10^{-5}$. 
We note that the error history of our proposed method decays fast at first and chooses to refine the time step for both equations in the partially explicit scheme. Then the algorithm gets back to coarse for equation \eqref{eq:partial_exp1} and still refines equation \eqref{eq:partial_exp2} for a while. Finally, it stabilizes to the coarse-coarse case at the latter part of the simulation. The number of refined coarse blocks is around 92/477. %, which is 34.6\% of the total number of coarse blocks. 
The average mean $L^2$ error across all time steps is $0.0104\%$ and the energy error is $0.0568\%$. As a reference, the fine-fine errors are $L^2$/energy errors are $0.0074\% /0.0323\%$, and the coarse-coarse errors are $0.0764\%/ 0.33041\%$, correspondingly.
\rev{Similar results are obtained in Figure \ref{fig:errs_ex1_type2_2} for the second type of error indicators, and in this case, $\delta_1 = \delta_2 = 5\times 10^{-6}$. We remark that, the thresholds need to be chosen differently for different types in order to get desirable results. We can see from Figure \ref{fig:errs_ex1_type1_2} and \ref{fig:errs_ex1_type2_2} that, when we use $\delta_1 = \delta_2 = 5\times 10^{-6}$ for both types, type 2 performs well, but the thresholds are too small for type 1 such that it is over-refined. If we use $\delta_1 = 1.5\times 10^{-4}$, $\delta_2 = 5\times 10^{-6}$ for both types, as seen from Figure \ref{fig:errs_ex1_type1} and \ref{fig:errs_ex1_type2}, type 1 performs well, but the thresholds are too large for type 2 such that the errors are closer to coarse + coarse case. }
%the transition from fine-fine to coarse-coarse is faster. 
During this finite-time simulation, our method outperforms the coarse-coarse method in terms of accuracy. Moreover, it converges to the fine-fine case fast and is \rev{computationally much cheaper}.

In the end, we show the mean errors when we choose different threshold parameters $\delta_1$, $\delta_2$ in Table \ref{tab:errs_ex1_deltadiff_type1}. We observe that as $\delta$ decreases, the errors are getting closer to the fine-fine case.

\begin{figure}
    \centering
    \includegraphics[width=1.0\textwidth]{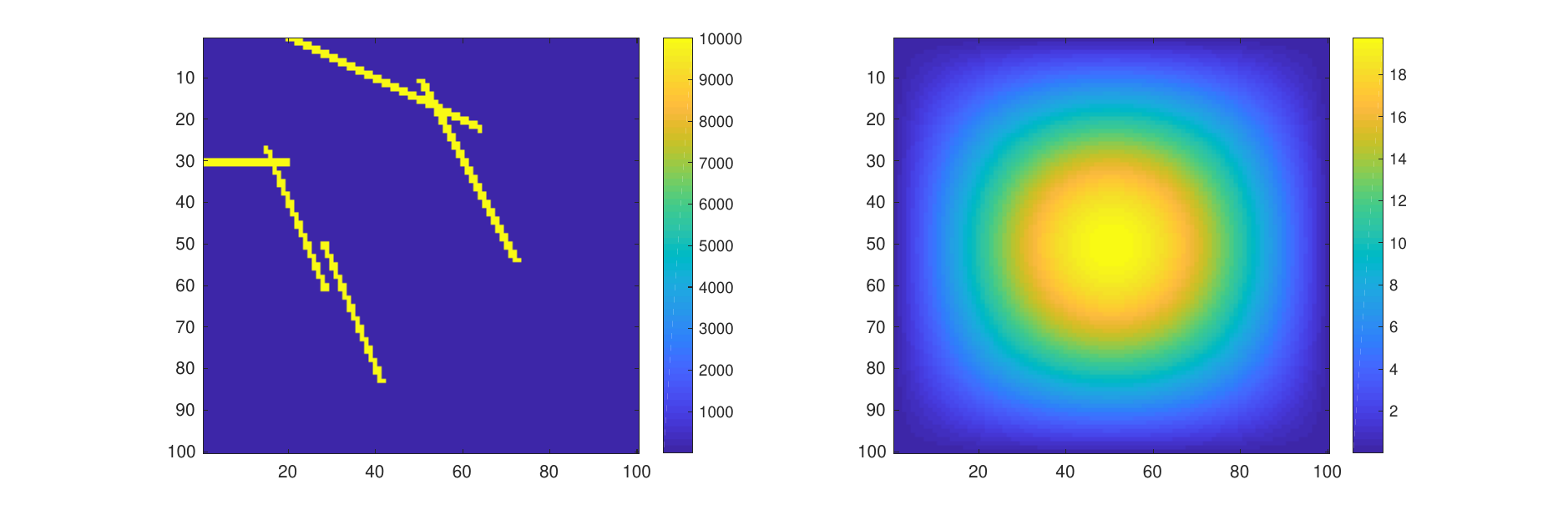}
    \caption{Example 1, left: permeability field, right: source term.}
    \label{fig:source_perm_ex1}
\end{figure}

\begin{figure}
    \centering
    \includegraphics[width=0.8\textwidth]{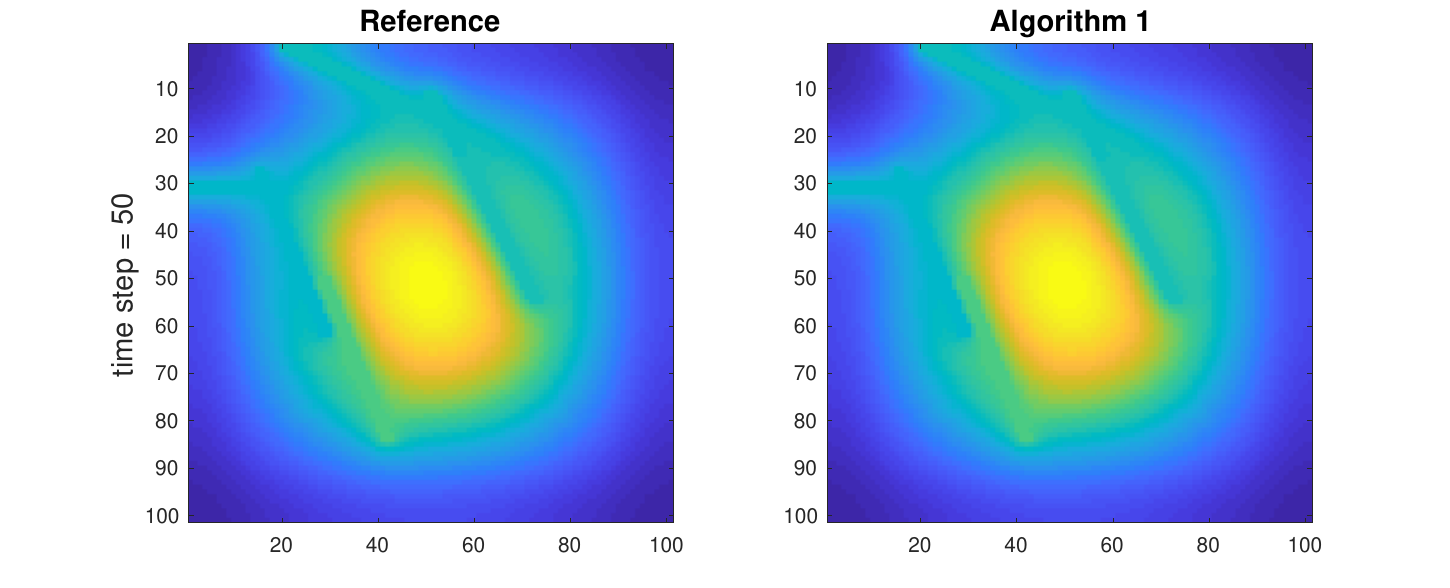}
    \includegraphics[width=0.8\textwidth]{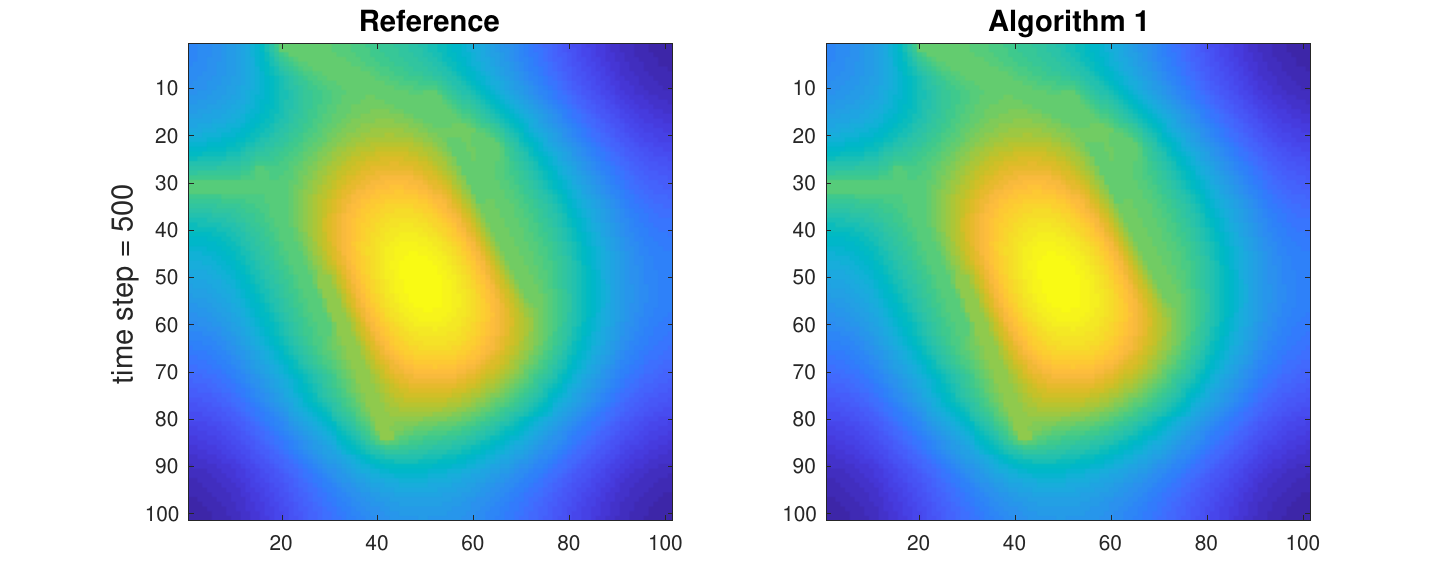}
    \caption{Example 1, the comparison of solutions at different time steps. Left: reference solutions, right: solutions obtained from the proposed algorithm.}
    \label{fig:sol_ex1}
\end{figure}

\begin{figure}
    \centering
    \includegraphics[width=1.0\textwidth]{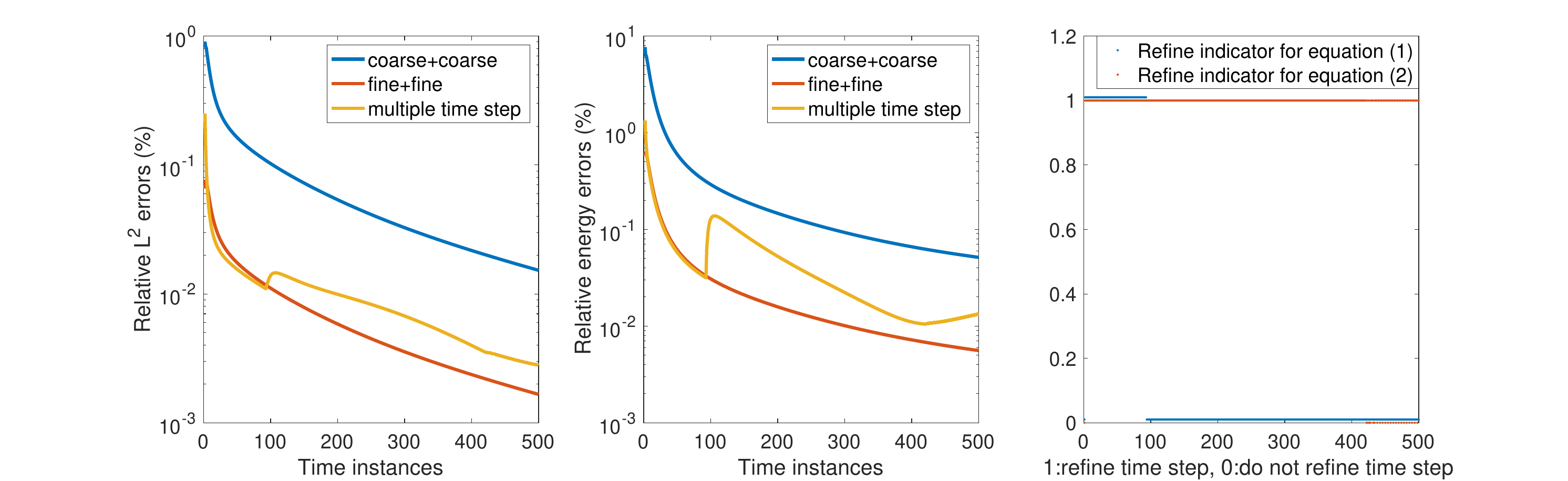}
    \caption{Example 1, using type 1 error indicators. \rev{$\delta_1 = 1.5\times 10^{-4}$, $\delta_2 = 5\times 10^{-6}$. Left and middle}: error history, right: refinement history for two equations. The number of refined steps for the first equation is 92, for the second equation is 477. \rev{The mean $L^2$ error is 0.0104\%, and the energy error is 0.0568\%.}}
    \label{fig:errs_ex1_type1} %eta1=eta2=5e-3
\end{figure}
          
% \begin{figure}
%     \centering
%     \includegraphics[width=1.0\textwidth]{pic/errs_ex1_type2-eps-converted-to.pdf}
%     \caption{Example 1, using type 2 error indicators \eqref{eq:type2}. Left: error history, right: refinement history for two equations. The number of refined steps for the first equation is 135, for the second equation is 136.}
%     \label{fig:errs_ex1_type2} %eta1=eta2=8e-7
% \end{figure}

% \begin{figure}
%     \centering
%     \includegraphics[width=1.0\textwidth]{pic/errs_ex1_alleps_type1-eps-converted-to.pdf}
%     \includegraphics[width=1.0\textwidth]{pic/errs_ex1_alleps_type2-eps-converted-to.pdf}
%     \caption{Example 1, error history with different error thresholds. Top: use type 1 error indicators, bottom: use type 2 error indicators.}
%     \label{fig:errs_ex1_deltadiff}
% \end{figure}

\begin{table}[]
    \centering
    \begin{tabular}{|l||*{3}{c|}}\hline
    \multicolumn{4}{|c|}{Mean errors ($L^2$/energy error)} \\ \hline
    \diagbox[width=5em]{$\delta_2$}{$\delta_1$}
    &\makebox[5em]{$1\times 10^{-4}$}&\makebox[5em]{$2\times 10^{-4}$}&\makebox[5em]{$5\times 10^{-4}$} \\\hline\hline
    $5\times 10^{-6}$ &(0.0086/ 0.0442)&(0.0113/ 0.0693) &(0.0199/ 0.1496)\\\hline
    $1\times 10^{-5}$ &(0.0109/ 0.0507)&(0.0113/ 0.0759) &(0.0220/ 0.155)\\\hline
    $1\times 10^{-4}$ &(0.0164/ 0.0627)&(0.0259/ 0.0979) &(0.0508/ 0.2069)\\\hline
    $5\times 10^{-4}$ &(0.0449/ 0.1520)&(0.0442/ 0.1512) &(0.0558/ 0.2091)\\\hline
        \multicolumn{4}{|c|}{\# of refinement steps (for eqn. \eqref{eq:partial_exp1}/ for eqn. \eqref{eq:partial_exp2})} \\ \hline
    \diagbox[width=5em]{$\delta_2$}{$\delta_1$}
    &\makebox[5em]{$1\times 10^{-4}$}&\makebox[5em]{$2\times 10^{-4}$}&\makebox[5em]{$5\times 10^{-4}$} \\\hline\hline
    $5\times 10^{-6}$ &(145, 477)&(62, 477) &(9, 477)\\\hline
    $1\times 10^{-5}$ &(145, 253)&(62, 254) &(9, 254)\\\hline
    $1\times 10^{-4}$ &(190, 146)&(75, 76) &(10, 16)\\\hline
    $5\times 10^{-4}$ &(191, 26)&(93, 26) &(19, 10)\\\hline
    \end{tabular}
    \caption{Example 1, top: average error over all time steps using type 1 error indicators with different error thresholds, the errors are in percentage; bottom: the number of refinement steps for equation \eqref{eq:partial_exp1}/ for equation \eqref{eq:partial_exp2}, respectively. References: fine-fine errors are 0.0074/0.0323; coarse-coarse errors are 0.0764/ 0.33041.}
    \label{tab:errs_ex1_deltadiff_type1}
\end{table}

\subsection{Example 2: time-independent singular source term}
In the second example, the configuration of the permeability field and the point source term $f(x,y)$ can be found in Figure \ref{fig:source_perm_ex2}. Similar to before, the conductivity is $10^4$ in the channel and $1$ in the background.

We set the total simulation time to be $T=0.02$. The number of coarse-scale time steps is $2000$ and the number of fine-scale time steps is $20000$. Again, the reference solutions are computed at a finer time scale with $200000$ steps using the Crank-Nicolson scheme. The comparison of \rev{solutions at different time steps} computed from different combinations of time scales is presented in Figure \ref{fig:sol_ex2}. 

The errors at coarse time instances are shown in Figure \ref{fig:errs_ex2_type1} when we use the first type of error indicators, similar behavior can be observed when we use the second type of error indicators. We can see that the errors of our proposed method decays fast and are similar to the fine-fine time step size for the partially explicit scheme. The refinement indicators in the right of Figure \ref{fig:errs_ex2_type1} demonstrate that to get comparable results, we only need $36$ refining steps for the implicit part and $639$ refining steps for the explicit part. 

%Similar as before, our proposed method demonstrate its efficiency and accuracy. 
%the algorithm automatically chooses different time step for the two equations \eqref{eq:partial_exp1} and \eqref{eq:partial_exp2}

\begin{figure}
    \centering
    \includegraphics[width=1.0\textwidth]{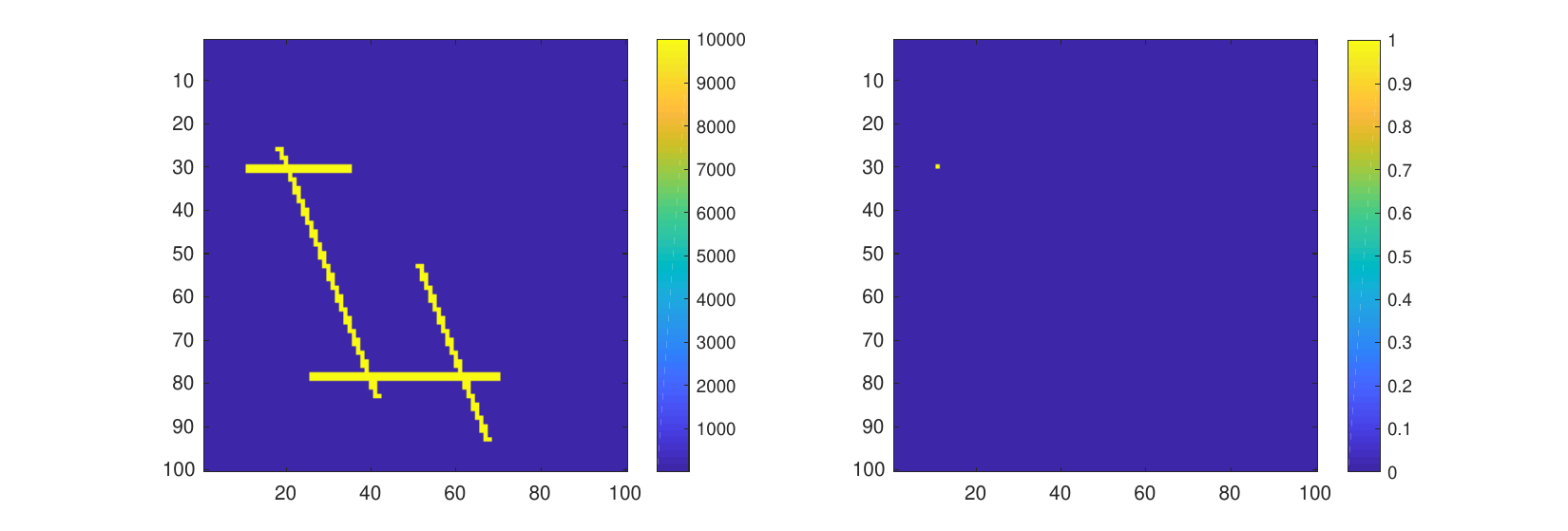}
    \caption{Example 2, left: permeability field, right: source term.}
    \label{fig:source_perm_ex2}
\end{figure}

\begin{figure}
    \centering
    \includegraphics[width=0.8\textwidth]{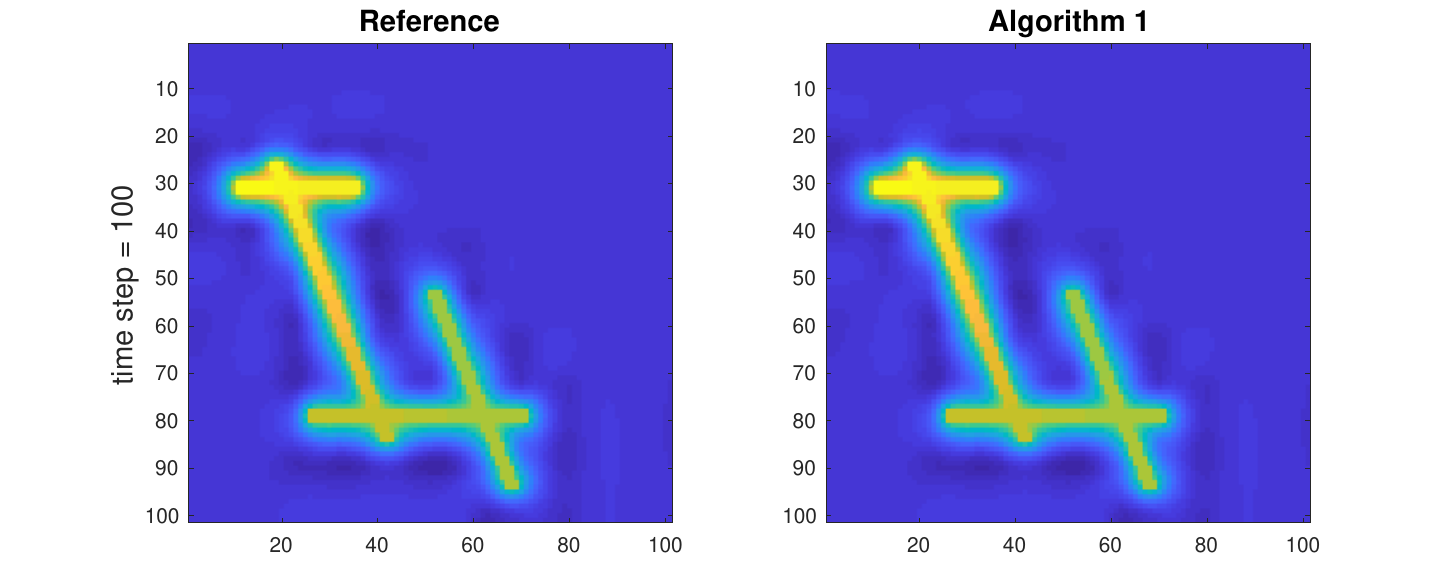}
    \includegraphics[width=0.8\textwidth]{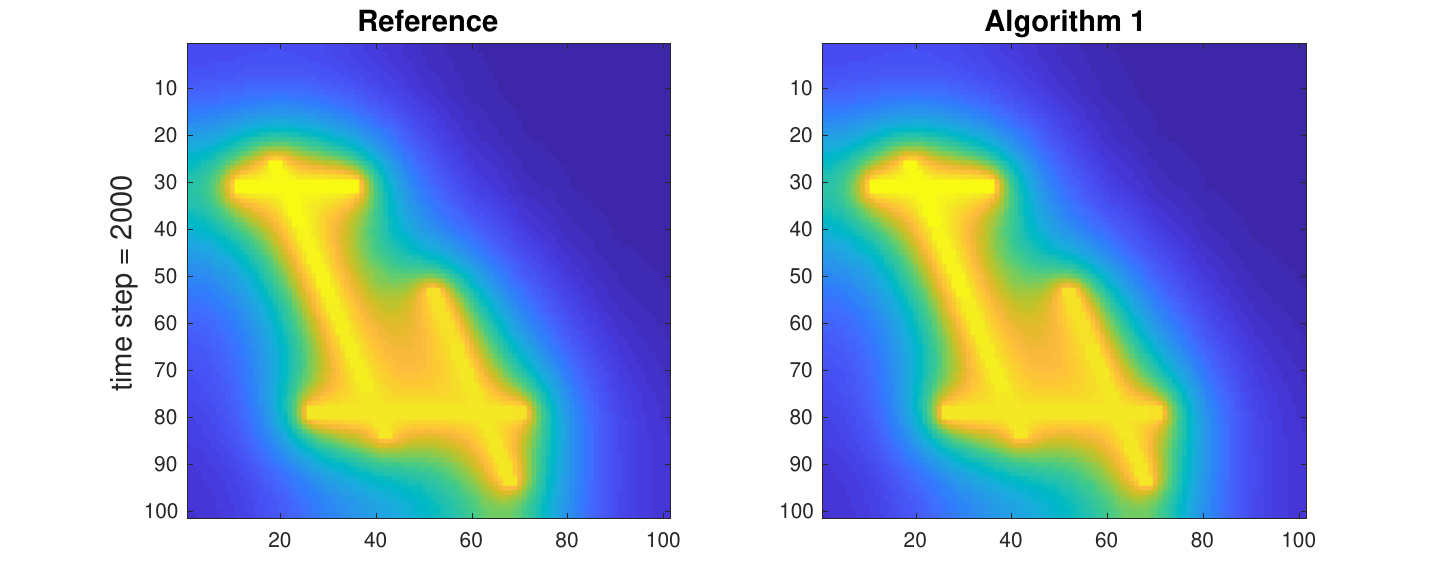}
    \caption{Example 2, the comparison of solutions at different time steps. Left: reference solutions, right: solutions obtained from the proposed algorithm.}
    \label{fig:sol_ex2}
\end{figure}

\begin{figure}
    \centering
    \includegraphics[width=1.0\textwidth]{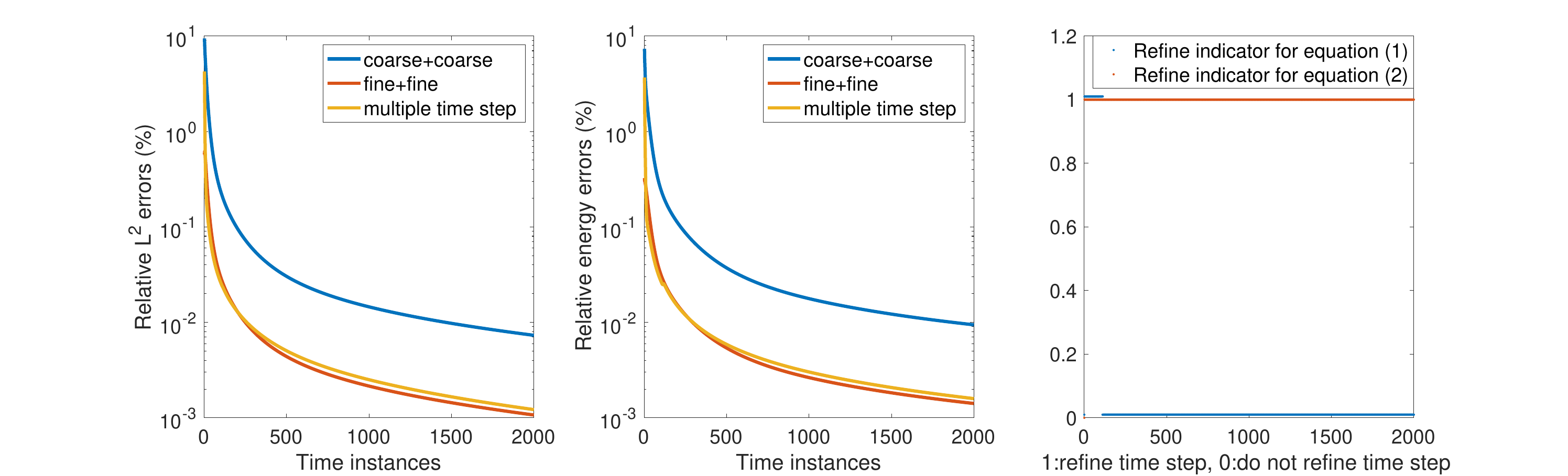} %ex2_err_eps1_5en9_eps2_1p5en11_type1_2
    \caption{Example 2, using type 1 error indicators. Left: error history, right: refinement history for two equations. The number of refined steps for the first equation is 111, for the second equation is 1999. The mean $L^2$ error is 0.0124\%, and the energy error is 0.0124\%. }
    \label{fig:errs_ex2_type1} %eta1=eta2=5e-3
\end{figure}

%   \begin{figure}
%     \centering
%     \includegraphics[width=1.0\textwidth]{pic/errs_ex2_type2_2en11-eps-converted-to.pdf}
%     \caption{Example 2, using type 2 error indicators \eqref{eq:type2}. Left: error history, right: refinement history for two equations. The number of refined steps for the first equation is 182, for the second equation is 182.}
%     \label{fig:errs_ex2_type2} %eta1=eta2=5e-3
% \end{figure}

% \begin{figure}
%     \centering
%     \includegraphics[width=1.0\textwidth]{pic/errs_ex2_alleps_type1-eps-converted-to.pdf}
%     \includegraphics[width=1.0\textwidth]{pic/errs_ex2_alleps_type2-eps-converted-to.pdf}
%     \caption{Example 2, error history with different error threshold. Top: use type 1 error indicators, bottom: use type 2 error indicators.}
%     \label{fig:errs_ex2_deltadiff}
% \end{figure}

In the end, we show the mean errors when we choose different threshold parameters $\delta_1$ and $\delta_2$ in Table \ref{tab:errs_ex2_deltadiff_type1} for type 1, and Table \ref{tab:errs_ex2_deltadiff_type2} for type 2. For both types of indicators, we observe that as $\delta$ decreases, the errors are getting closer to the fine-fine case, and one needs more refinement steps. Note that for the results in Table \ref{tab:errs_ex2_deltadiff_type2}, we use the same thresholds for $\delta_1$ and $\delta_2$. The numerical results show that for the first type of indicators, the refinement of the two equations may not be carried out simultaneously. However, for the second type of error indicator, the refinement of the two equations is consistent. Moreover, in both types, the errors are more sensitive to the refinement in the second equation (explicit part).

 \begin{table}[]
    \centering
    \begin{tabular}{|l||*{3}{c|}}\hline
    \multicolumn{4}{|c|}{Mean errors ($L^2$/energy error)} \\ \hline
    \diagbox[width=5em]{$\delta_2$}{$\delta_1$}
    &\makebox[5em]{$5\times 10^{-9}$}&\makebox[5em]{$7.5\times 10^{-9}$}&\makebox[5em]{$1\times 10^{-8}$} \\\hline\hline
    $1\times 10^{-12}$ &(0.0135/ 0.0126)&(0.0147/ 0.0136) &(0.0273/ 0.0227)\\\hline
    $1\times 10^{-11}$ &(0.0145/ 0.0138)&(0.0157/ 0.0148)&(0.0281/ 0.0238)\\\hline
    $3\times 10^{-11}$ &(0.0212/ 0.0217)&(0.0223/ 0.0226)&(0.0339/ 0,0308)\\\hline
    $5\times 10^{-11}$ &(0.0244/ 0.0254)&(0.0254/ 0.0263)&(0.0364/ 0.0339)\\\hline
        \multicolumn{4}{|c|}{\# of refinement steps (for eqn. \eqref{eq:partial_exp1}, for eqn. \eqref{eq:partial_exp2})} \\ \hline
    \diagbox[width=5em]{$\delta_2$}{$\delta_1$}
    &\makebox[5em]{$5\times 10^{-9}$}&\makebox[5em]{$7.5\times 10^{-9}$}&\makebox[5em]{$1\times 10^{-8}$} \\\hline\hline
    $1\times 10^{-12}$ &(37,1999) &(24,1999) &(16,1999)\\\hline
    $1\times 10^{-11}$ &(37,1120) &(24,1120) &(16,1120)\\\hline
    $3\times 10^{-11}$ &(37,279) &(24,279) &(16,278)\\\hline
    $5\times 10^{-11}$ &(37,166) &(24,167) &(16,167)\\\hline
    \end{tabular}
    \caption{Example 2, top: average error over all time steps using type 1 error indicators with different error thresholds, the errors are in percentage; bottom: the number of refinement steps for equation \eqref{eq:partial_exp1}/ for equation \eqref{eq:partial_exp2}, respectively. References: fine-fine errors are 0.0122/0.0091 (\%); coarse-coarse errors are 0.1243/0.0921 (\%).}
    \label{tab:errs_ex2_deltadiff_type1}
\end{table}

 \begin{table}[]
    \centering
    \begin{tabular}{|l||*{5}{c|}}\hline
    $\delta_1 = \delta_2$ (in $\cdot \times 10^{-11}$)
     & {$1$} & {$1.5$}  & {$2$} & {$3$}  & {$5$} \\\hline\hline %& {$8$}
     Mean $L^2$ errors &0.0149  &0.0174  &0.0196  &0.0230  &0.0278 \\\hline %&0.0331
    Mean energy errors &0.0142 &0.0172 &0.0198 &0.0237 &0.0291 \\\hline % &0.0345
    \# of refinement steps &(96, 984) & (69,588) &(45, 372) &(31, 217) &(24,113)\\\hline % &(19,63)
    \end{tabular}
    \caption{Example 2, using type 2 error indicators with different error thresholds. The average errors (in percentage) over all time steps, and the number of refinement steps for equation \eqref{eq:partial_exp1}, for equation \eqref{eq:partial_exp2}, respectively.}
    \label{tab:errs_ex2_deltadiff_type2}
\end{table}

% \begin{figure}
%     \centering
%     \includegraphics[width=1.0\textwidth]{pic/sol_ex2-eps-converted-to.pdf}
%     \caption{Example 2, top left: reference solution, top right: fine-fine time steps, bottom left: coarse+coarse time steps, bottom right: proposed algorithm .}
%     \label{fig:sol_ex2}
% \end{figure}

% \begin{figure}
%     \centering
%     \includegraphics[width=1.0\textwidth]{pic/errs_ex2-eps-converted-to.pdf}
%     \caption{Example 2, left: error history, right: refinement history for two equations.}
%     \label{fig:errs_ex2}
% \end{figure}

\subsection{Example 3: time-dependent discontinuous source term}
In the last example, we consider a point source term where the location of the point changes during the simulation. The total simulation time to be $T=0.2$. In the first half of the time interval, $f(x,y) = 1$ at $(x,y)=(0.3,0.5)$ and $f(x,y) = 0$ elsewhere. In the second half of the time interval, $f(x,y) = 1$ at $(x,y)=(0.3,0.11)$ and $f(x,y) = 0$ elsewhere. The number of coarse scale time steps is still $2000$ and the number of fine scale time steps is $20000$. The permeability is the same as in Example 2. 

The comparison of solutions at different time steps using different schemes is presented in Figure \ref{fig:sol_ex3}. 

The errors at coarse time instances are shown in Figure \ref{fig:errs_ex3_type1} for the first type of the indicators, and the behavior for the second type is similar as before, so we omit the results in this example. We can see that at the beginning, \rev{the scheme requires refinement}. Then when the errors \rev{get} smaller, the refinement is deactivated and we can obtain good results without refinement. In the middle of the simulation, when the discontinuity of the source term occurs, there is a jump in the error, and the indicators change from ``not refine" to ``refine" automatically. Then similar processes happen. This shows our algorithm is reliable for complicated source terms.
% Similar behaviors of the errors v.s. threshold parameters are presented in Figure \ref{fig:errs_ex3_deltadiff}. 

We also present the mean errors when we choose different threshold parameters $\delta_1$ and $\delta_2$ in Table \ref{tab:errs_ex3_deltadiff_type1} for type 1, and Table \ref{tab:errs_ex3_deltadiff_type2} for type 2. In this example, we observe again that only a few refinement steps are needed in the first equation (implicit part), and the errors are more sensitive to the refinement in the second equation. Using around $3$ refinement steps in the first equation and $470$ steps in the second equation, the average errors are already close to the fine-fine case. 

\rev{We remark that in the above numerical examples, the average run-time of one time step using an implicit scheme without splitting is around 0.0013 second with MATLAB direct solver, and the average run-time of one time step using our splitting method with partially explicit scheme is around 0.0009 second.} This completes the numerical section.

\begin{figure}
    \centering
    \includegraphics[width=0.8\textwidth]{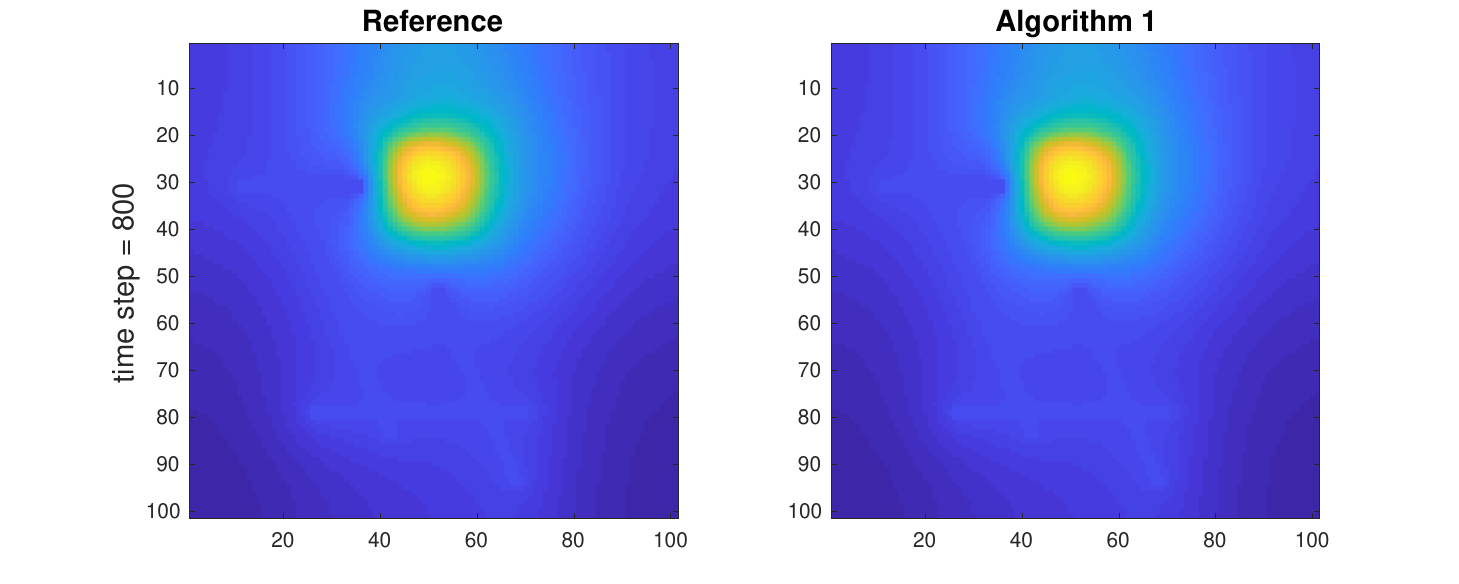}
    \includegraphics[width=0.8\textwidth]{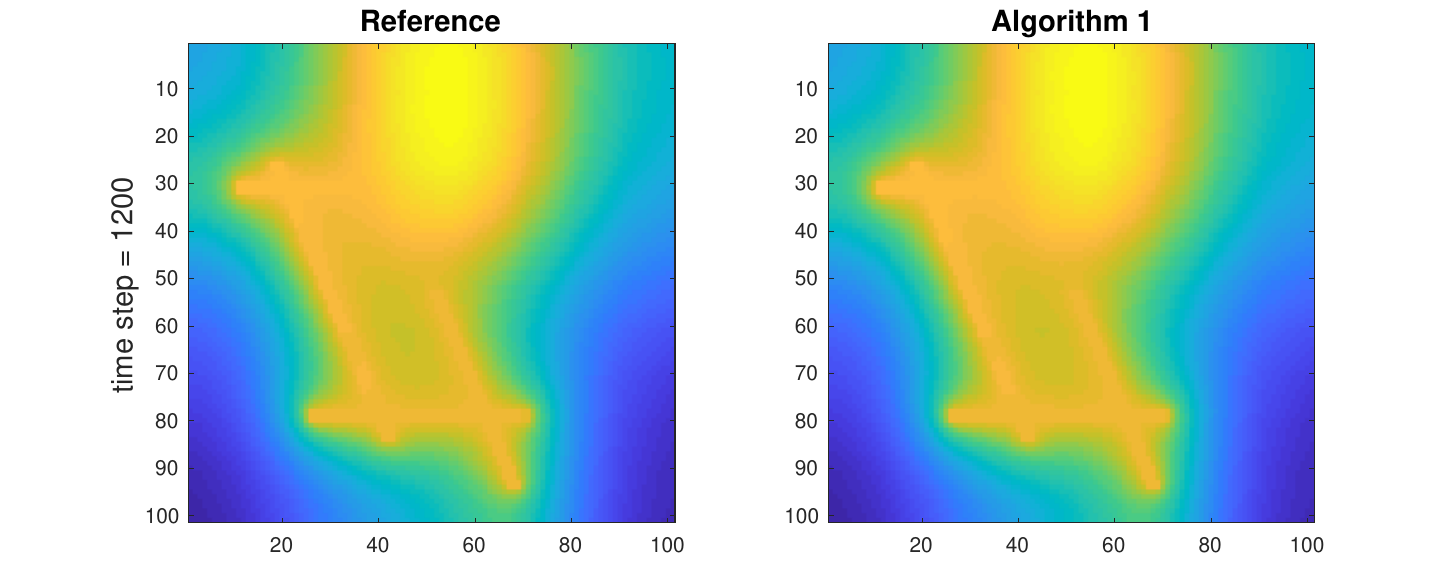}
    \includegraphics[width=0.8\textwidth]{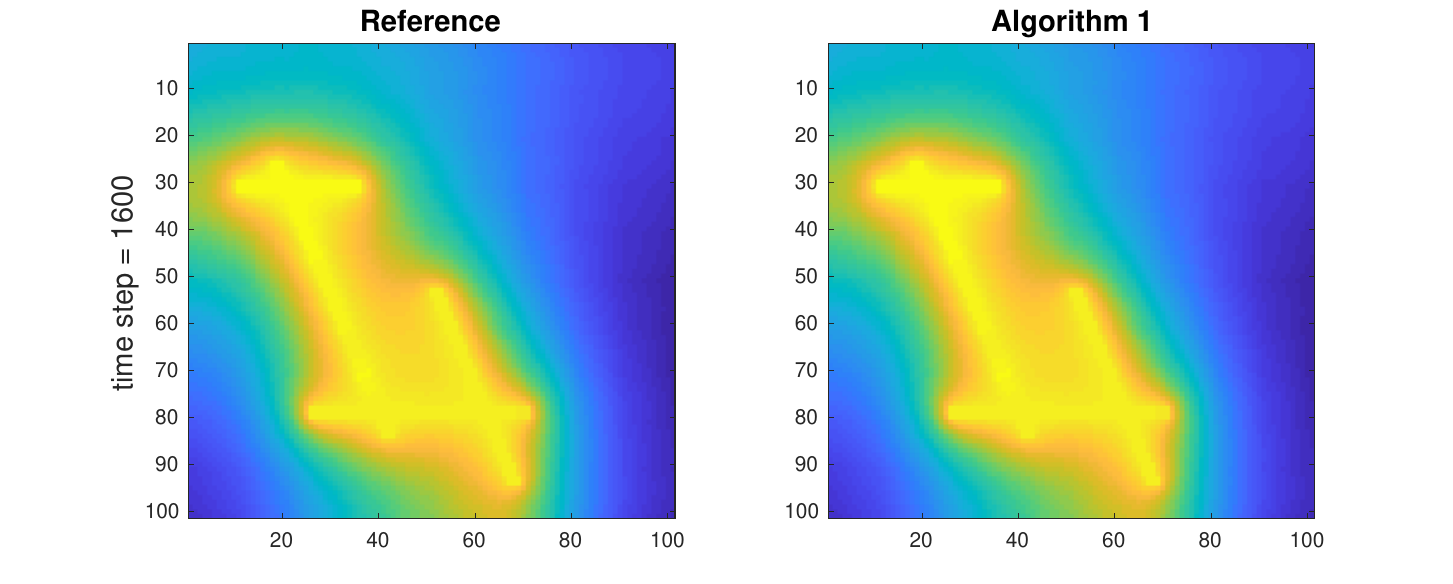}
    \caption{Example 3, the comparison of solutions at different time steps.}
    \label{fig:sol_ex3}
\end{figure}
       
% \begin{figure}
%     \centering
%     \includegraphics[width=1.0\textwidth]{pic_new/ex3_eps11en8_eps25en10-eps-converted-to.pdf}
%     \caption{\rev{Example 3, left: error history, right: refinement history for two equations using the first type of indicators. The number of refined steps for the first equation is 638, for the second equation is 1577. The mean $L^2$ error is 0.02\%, and the energy error is 0.077\%.}}
%     \label{fig:errs_ex3_type1}
% \end{figure}    

\begin{figure}
    \centering
    \includegraphics[width=1.0\textwidth]{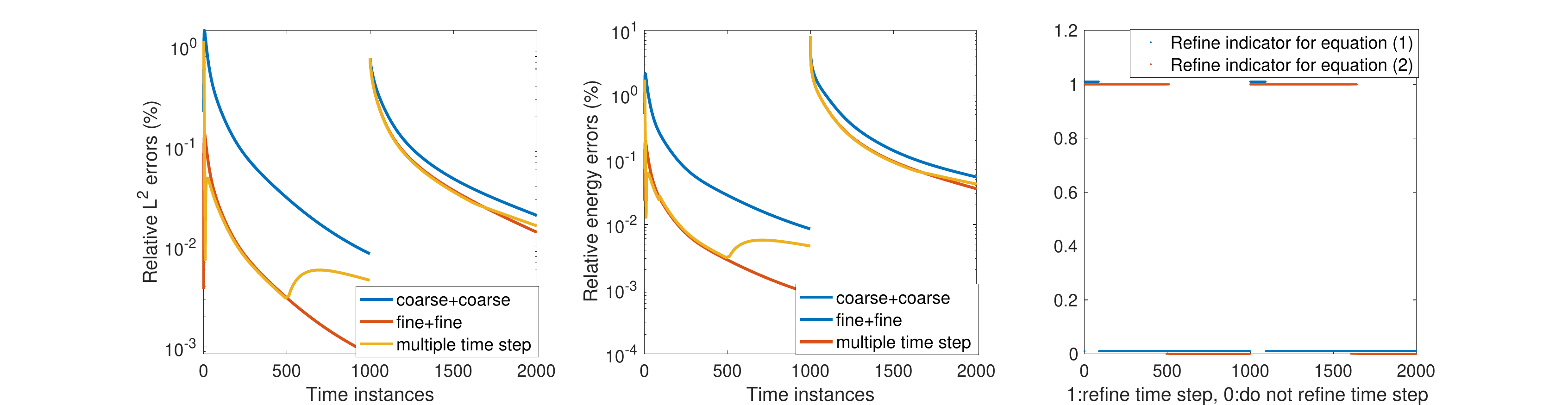}%{pic/ex3_err_eps1_1en8_eps2_1en10_type1-eps-converted-to.pdf}
    \caption{\rev{Example 3, left: error history, right: refinement history for two equations using the first type of indicators. The number of refined steps for the first equation is 180, for the second equation is 1137. The mean $L^2$ error is 0.041\%, and the energy error is 0.132\%.}}
    \label{fig:errs_ex3_type1}
\end{figure}    
          
% \begin{figure}
%     \centering
%     \includegraphics[width=1.0\textwidth]{pic/ex3_err_eps1_3en7_eps2_3en10_type1_2-eps-converted-to.pdf}
%     \caption{Example 3, left: error history, right: refinement history for two equations using the first type of indicators. The number of refined steps for the first equation is 3, for the second equation is 470. \rev{The mean $L^2$ error is 0.0461\%, and the energy error is 0.1416\%.}}
%     \label{fig:errs_ex3_type1}
% \end{figure}

% \begin{figure}
%     \centering
%     \includegraphics[width=1.0\textwidth]{pic/errs_ex3_alleps_type1_v2-eps-converted-to.pdf}
%     \includegraphics[width=1.0\textwidth]{pic/errs_ex3_alleps_type2-eps-converted-to.pdf}
%     \caption{Example 3, error history with different error thresholds. Top: use type 1 error indicators, bottom: use type 2 error indicators.}
%     \label{fig:errs_ex3_deltadiff}
% \end{figure}

 \begin{table}[]
    \centering
    \begin{tabular}{|l||*{3}{c|}}\hline
    \multicolumn{4}{|c|}{Mean errors ($L^2$/energy error)} \\ \hline
    \diagbox[width=5em]{$\delta_2$}{$\delta_1$}
    &\makebox[5em]{$1\times 10^{-8}$}&\makebox[5em]{$2\times 10^{-8}$}&\makebox[5em]{$3\times 10^{-7}$} \\\hline\hline
    % $5\times 10^{-11}$ &0.0403/ 0.1311  &0.0402/ 0.1312  &0.0402/0.1315 \\\hline
    $3\times 10^{-10}$ &0.0461/ 0.1412  &0.0461/ 0.1413  & 0.0461/ 0.1416 \\\hline
    $1\times 10^{-9}$  &0.0525/ 0.1532  &0.0557/  0.1595  & 0.0557/  0.1597 \\\hline
    $5\times 10^{-9}$  &0.0621/0.1706  &0.0622/  0.1715  &0.0769/  0.1988 \\\hline
    $1\times 10^{-8}$  &0.0883/0.2162  &0.0883/  0.2162  &0.0855/  0.2153 \\\hline
        \multicolumn{4}{|c|}{\# of refinement steps (for eqn. \eqref{eq:partial_exp1}, for eqn. \eqref{eq:partial_exp2})} \\ \hline
    \diagbox[width=5em]{$\delta_2$}{$\delta_1$}
    &\makebox[5em]{$1\times 10^{-8}$}&\makebox[5em]{$2\times 10^{-8}$}&\makebox[5em]{$3\times 10^{-7}$} \\\hline\hline
    % $5\times 10^{-11}$ &(180,1999)  &(108,1999)  &(3, 1999) \\\hline
    $3\times 10^{-10}$ &(180,470)  &(108, 470)  &(3, 470) \\\hline
    $1\times 10^{-9}$  &(209, 248)  &(167, 182)  &(3,182) \\\hline
    $5\times 10^{-9}$  &(265, 112)  &(167, 108)  &(3, 32) \\\hline
    $1\times 10^{-8}$  &(265, 14 )  &(168, 14)  &(6, 13) \\\hline
    \end{tabular}
    \caption{Example 3, top: average error over all time steps using type 1 error indicators with different error thresholds, the errors are in percentage; bottom: the number of refinement steps for equation \eqref{eq:partial_exp1}, for equation \eqref{eq:partial_exp2}, respectively. References: fine-fine errors are 0.0403/0.1308 (\%); coarse-coarse errors are 0.0950/0.2323 (\%).}
    \label{tab:errs_ex3_deltadiff_type1}
\end{table}

   \begin{table}[]
    \centering
    \begin{tabular}{|l||*{5}{c|}}\hline
    $\delta_1 = \delta_2$ (in $\cdot \times 10^{-10}$)
     & {$1$} & {$3$}  & {$5$} & {$7$}  & {$10$} \\\hline\hline
     Mean $L^2$ errors &0.0412   & 0.0464  &0.0501 &0.0529 &0.0565\\\hline
    Mean energy errors &0.1322   &0.1402   & 0.1489  &0.1544 &0.1612\\\hline
    \# of refinement steps &(264,1094) &(89,451) &(43,303) &(25,232) &(13,170) \\\hline 
    \end{tabular}
    \caption{Example 3, using type 2 error indicators with different error thresholds. The average errors (in percentage) over all time steps, and the number of refinement steps for equation \eqref{eq:partial_exp1}, for equation \eqref{eq:partial_exp2}, respectively.}
    \label{tab:errs_ex3_deltadiff_type2}
\end{table}

\section{Conclusion} \label{sec:conclusion}
We presented a multirate method and an adaptive algorithm with some error estimators to solve parabolic equations with multiscale diffusivity coefficients satisfying the accuracy requirement and at a reduced computational cost. We first constructed some multiscale spaces based on CEM-GMsFEM and NLMC, and then adopted appropriate multirate temporal splitting schemes. To be specific, the degrees of freedom corresponding to the fast component (the high permeable regions) are handled implicitly, here the dimension of the multiscale subspace is small. Then the multiscale basis that corresponds to the slow flow are constructed and this part is treated explicitly. We started with a coarse time step size for both implicit and explicit parts, and estimated the errors using some locally computable estimators to determine whether the temporal mesh needs to be refined. The process is carried out adaptively. \rev{Several} numerical tests were performed. The results showed that with reduced computational effort, we can get reliable and accurate approximations. Currently, we use a two-level time step size, future work includes the development of multiple level schemes and space-time adaptive algorithms.

\appendix

\section{Spatial convergence}\label{appendix:thm}

In this section, we give an estimate of the spatial error for the semi-discretization system.

Denote by $V=H_0^1(\Omega)$. Let $u(t,\cdot) \in V$ be the solution of 
\begin{equation}\label{eq:ref_weak}
\begin{aligned}
    (\frac{\partial u }{\partial  t}, v) + a(u,v) &= (f,v), \\
    u(0,\cdot) &= u_0, 
\end{aligned}
\end{equation}
for all $v\in V$, $t\in(0,T]$.

% Let $u_{\text{glo}}\in V_{\text{glo}}$ be the solution of 
% \begin{equation}\label{eq:glo_weak}
% \begin{aligned}
%     (\frac{\partial u_{\text{glo}} }{\partial  t}, v) + a(u_{\text{glo}},v) &= (f,v), \\
%     u_{\text{glo}}(0,\cdot) &= u_0, 
% \end{aligned}
% \end{equation}
% for all $v\in V_{\text{glo}}$, $t\in(0,T]$.

Furthermore, let $V_H=\text{span} \{\psi_{m}^{(i)}, \; 0 \leq m \leq  m_i, 0 \leq i \leq N_c\}$ be the localized NLMC space where $\psi_{m}^{(i)}$ are defined in \eqref{eq:basis}, and $u_H \in V_H$ be the solution of
\begin{equation}\label{eq:nlmc_weak}
\begin{aligned}
    (\frac{\partial u_H }{\partial  t}, v) + a(u_H,v) &= (f,v), \\
    u_H(0,\cdot) &= u_0, 
\end{aligned}
\end{equation}
for all $v\in V_H$, $t\in(0,T]$.

Let $\psi_{m,\text{glo}}^{(i)} \in V_0(\Omega)$ be the global NLMC basis, which are computed form the following constraint energy minimizing problem
\begin{equation}\label{eq:glo_basis}
\begin{aligned}
a(\psi_{m,\text{glo}}^{(i)}, v) + \sum_{K_j \subset \Omega} \left(\mu_0^{(j)} \int_{K_{j,m}}v  + \sum_{1 \leq n \leq m_j} \mu_n^{(j)} \int_{ f_n^{(j)}} v \right) = 0, \quad   \forall v\in V_0(K_i^+), &\\
\int_{K_{j,m}}\psi_{m,\text{glo}}^{(i)}    = \delta_{ij} \delta_{m0}, \quad   \forall K_j \subset \Omega,& \\
\int_{f_n^{(j)}} \psi_{m,\text{glo}}^{(i)}   = \delta_{ij} \delta_{mn}, \quad   \forall f_n^{(j)} \in \mathcal{F}_{j}, \;  \forall K_j \subset\Omega,&
\end{aligned}
\end{equation}
where $\mu_0^{(j)}, \mu_n^{(j)} \in \mathbb{R}$ are Lagrange multipliers.

The global NLMC space is then defined as $V_{\text{glo}} = \text{span} \{\psi_{m,\text{glo}}^{(i)}, \; 0 \leq m \leq  m_i, 0 \leq i \leq N_c\}$. Let $R_{\text{glo}}: V\rightarrow V_{\text{glo}}$ be the projection operator such that for any $u\in V$
\begin{equation*}
    a(R_{\text{glo}},v) = a(u,v),\quad \forall v \in V_{\text{glo}},
\end{equation*}
and $R_H: V\rightarrow V_H$ be the operator such that for any $u\in V$
\begin{equation*}
    a(R_H,v) = a(u,v),\quad \forall v \in V_H.
\end{equation*}

Define the operator $\mathcal{A}: D(\mathcal{A}) \rightarrow L^2(\Omega)$ such that for any $u\in D(\mathcal{A})$, 
\begin{equation*}
    (\mathcal{A}u,v) = a(u,v),\quad \forall v \in V.
\end{equation*}

Following a similar proof as presented in \cite{zhao2020analysis, tony-nlmc}, we have the following lemmas 
\begin{lemma}
Let $u\in D(\mathcal{A})$, then we have
\begin{equation}
   \begin{aligned}
    \|u-R_{\text{glo}} u\|_a &\leq CH\|\mathcal{A}u\|_{L^2(\kappa^{-1})},\\
        \|u-R_{\text{glo}} u\| &\leq CH^2\kappa_{\text{min}}^{-\frac{1}{2}} \| \mathcal{A}u\|_{L^2(\kappa^{-1})}.
       \end{aligned}
\end{equation}
where $\|v\|_{L^2(\kappa^{-1})}^2 = \int_{\Omega} \kappa^{-1} v^2$.
\end{lemma}

\begin{lemma}
If the oversampling size is in $O(log(\frac{\max(\kappa)}{H}))$, then we have \begin{equation}
    \begin{aligned}
        \|u-R_H u\|_a &\leq CH\|\mathcal{A}u\|_{L^2(\kappa^{-1})},\\
         \|u-R_H u\| &\leq CH^2\kappa_{\text{min}}^{-\frac{1}{2}} \| \mathcal{A}u\|_{L^2(\kappa^{-1})}.
    \end{aligned}
\end{equation}
\end{lemma}

\begin{theorem}
    Let $u(t,\cdot)$ be the solution of \eqref{eq:ref_weak} and $u_H(t,\cdot)$ be the solution of \eqref{eq:nlmc_weak}, we have
    \begin{equation*}
        \|u(T, \cdot) - u_H(T, \cdot) \|^2 + \int_0^T \|u-u_H\|_a^2 dt \leq C\kappa_{\text{min}}^{-1} H^2 \left(\|u_0 \|^2 + \int_0^T \|f\|^2 dt \right).
    \end{equation*}
\end{theorem}

\begin{proof}
Take $v= \frac{\partial u }{\partial  t}$ in \eqref{eq:ref_weak} and integrate over $(0,T)$, we get
\begin{equation*}
\begin{aligned}
        \int_0^T \|\frac{\partial u }{\partial  t}\|^2 dt  + \frac{1}{2} \int_0^T \frac{d}{d  t}\|u\|_a^2 dt = \int_0^T (f, \frac{\partial u }{\partial  t}) dt \leq \frac{1}{2} \int_0^T \|f\|^2 dt +  \frac{1}{2} \int_0^T \|\frac{\partial u }{\partial  t}\|^2 dt, 
\end{aligned}
\end{equation*}
this implies 
\begin{equation}\label{eq:ubound}
\begin{aligned}
       \frac{1}{2} \int_0^T \|\frac{\partial u }{\partial  t}\|^2 dt +  \frac{1}{2} \|u(T, \cdot)\|_a^2  \leq C ( \|u_0\|_a^2 + \int_0^T \|f\|^2 dt).
\end{aligned}
\end{equation}
Similarly, take $v= \frac{\partial u_H }{\partial  t}$ in \eqref{eq:nlmc_weak} and integrate over $(0,T)$, we have 
\begin{equation}\label{eq:uHbound}
\begin{aligned}
       \frac{1}{2} \int_0^T \|\frac{\partial u_H }{\partial  t}\|^2 dt +  \frac{1}{2} \|u_H(T, \cdot)\|_a^2  \leq C ( \|u_0\|_a^2 + \int_0^T \|f\|^2 dt).
\end{aligned}
\end{equation}

On the other hand, from equations \eqref{eq:ref_weak} and \eqref{eq:nlmc_weak}, we have 
\begin{equation*}
    (\frac{\partial (u-u_H) }{\partial  t}, v) + a(u-u_H,v) = 0,
\end{equation*}
for all $v\in V_H$. Then we have
\begin{equation*}
        (\frac{\partial (u-u_H) }{\partial  t}, u_H) + a(u-u_H,u_H)  =      (\frac{\partial (u-u_H) }{\partial  t}, R_H u) + a(u-u_H,R_H u)=0.
\end{equation*}
Thus
\begin{equation*}
    \begin{aligned}
           &\frac{1}{2} \frac{d}{d  t}\|u-u_H\|^2  + \|u-u_H\|_a^2  = (\frac{\partial (u-u_H) }{\partial  t}, u-R_H u) + a(u-u_H,u-R_H u)\\
           %&\leq \|\frac{\partial (u-u_H) }{\partial  t}\| \| u-R_H u\| + \|u-u_H\|_a \|u-R_H uu\|_a\\
           &\leq \left( \|\frac{\partial u }{\partial  t}\|+\|\frac{\partial u_H }{\partial  t}\|\right)\| u-R_H u\| + \frac{1}{2}\|u-u_H\|_a^2 + \frac{1}{2}\|u-R_H u\|_a^2,
    \end{aligned}
\end{equation*}
integrate over $(0,T)$, we get
\begin{equation*}
    \begin{aligned}
           &\frac{1}{2} \|u(T, \cdot)-u_H(T, \cdot)\|^2  + \frac{1}{2} \int_0^T  \|u-u_H\|_a^2 dt\\
           &\leq \int_0^T  \left( \|\frac{\partial u }{\partial  t}\|+\|\frac{\partial u_H }{\partial  t}\|\right)\| u-R_H u\| dt  + \frac{1}{2}\int_0^T \|u-R_H u\|_a^2dt\\
           &\leq \left( \int_0^T ( \|\frac{\partial u }{\partial  t}\|+\|\frac{\partial u_H }{\partial  t}\|)^2 dt \right)^{\frac{1}{2}} \left( \int_0^T CH^4\kappa_{\text{min}}^{-1}\|\mathcal{A}u\|_{L^2(\kappa^{-1})}^2  dt \right)^{\frac{1}{2}}\\
           &+\frac{1}{2}\int_0^T CH^2\|\mathcal{A}u\|_{L^2(\kappa^{-1})}^2 dt\\
           &\leq C\kappa_{\text{min}}^{-1} H^2 ( \|u_0\|_a^2 + \int_0^T \|f\|^2 dt),
    \end{aligned}
\end{equation*}
using \eqref{eq:ubound},\eqref{eq:uHbound} and the fact that $\mathcal{A}u = f- \frac{\partial u }{\partial  t} $. This completes the proof.

\end{proof}

\section{Comparison of results between two types of indicators}\label{appendix:indicator}

\rev{In the following, we consider example 1, and present the results for both type 1 and type 2 indicators when $\delta_1 = \delta_2 = 5\times 10^{-6}$. Additionally, we also show results for type 2 when $\delta_1 = 1.5\times 10^{-4}$, $\delta_2 = 5\times 10^{-6}$, this is to compare with the results in Figure \ref{fig:errs_ex1_type1} with the same sets of thresholds.}
\begin{figure}
    \centering
    \includegraphics[width=1.0\textwidth]{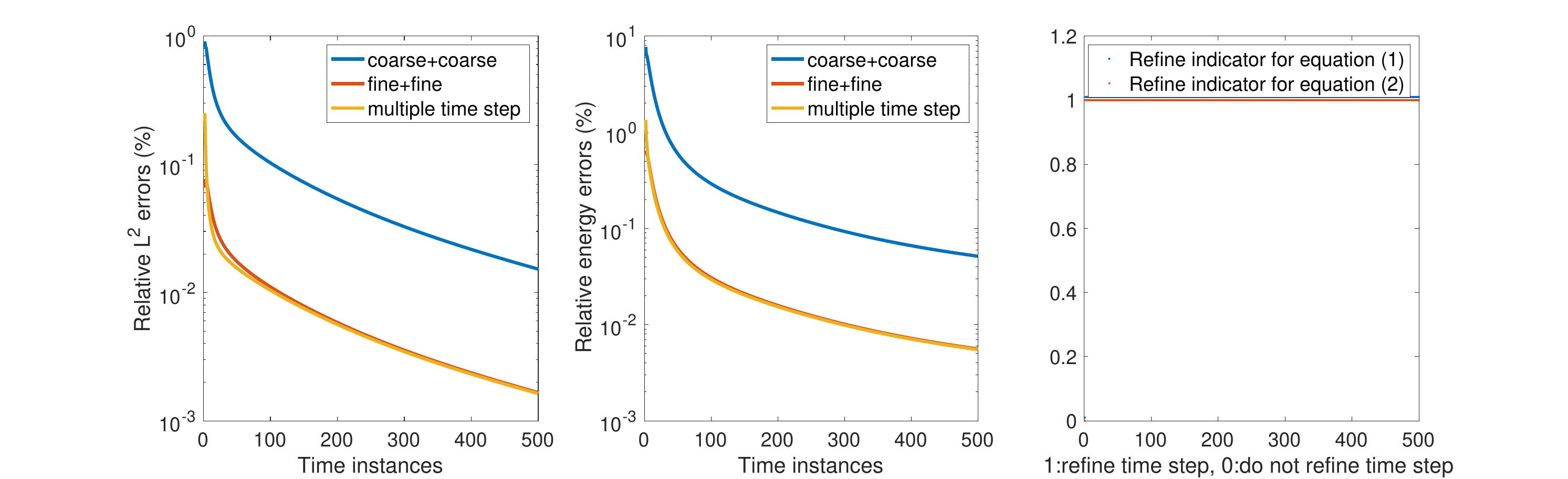}
    \caption{\rev{Example 1, using type 1 error indicators. $\delta_1 = \delta_2 = 5\times 10^{-6}$. Left and middle: error history, right: refinement history for two equations. The number of refined steps for the first equation is 499, for the second equation is 499. The mean $L^2$ error is 0.007\%, and the energy error is 0.0335\%.}}
    \label{fig:errs_ex1_type1_2} %eta1=eta2=5e-3
\end{figure}

\begin{figure}
    \centering
    \includegraphics[width=1.0\textwidth]{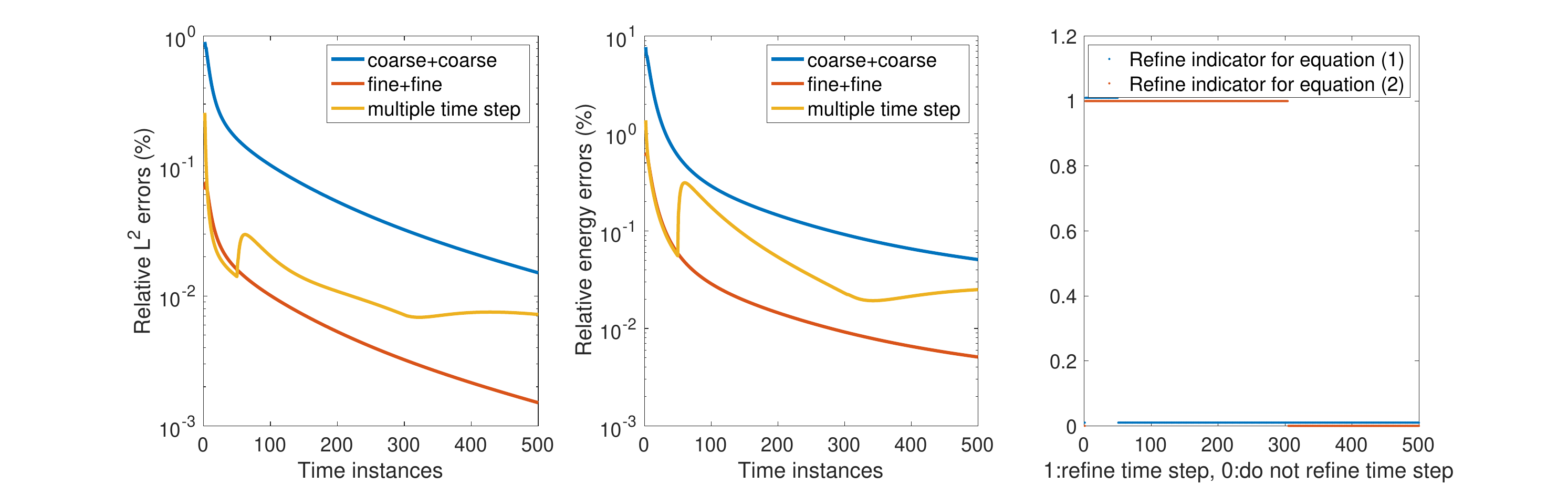}
    \caption{\rev{Example 1, using type 2 error indicators. $\delta_1 = \delta_2 = 5\times 10^{-6}$. Left and middle: error history, right: refinement history for two equations. The number of refined steps for the first equation is 49, for the second equation is 303. The mean $L^2$ error is 0.0132\%, and the energy error is 0.0812\%.}}
    \label{fig:errs_ex1_type2_2} %eta1=eta2=5e-3
\end{figure}

\begin{figure}
    \centering
    \includegraphics[width=1.0\textwidth]{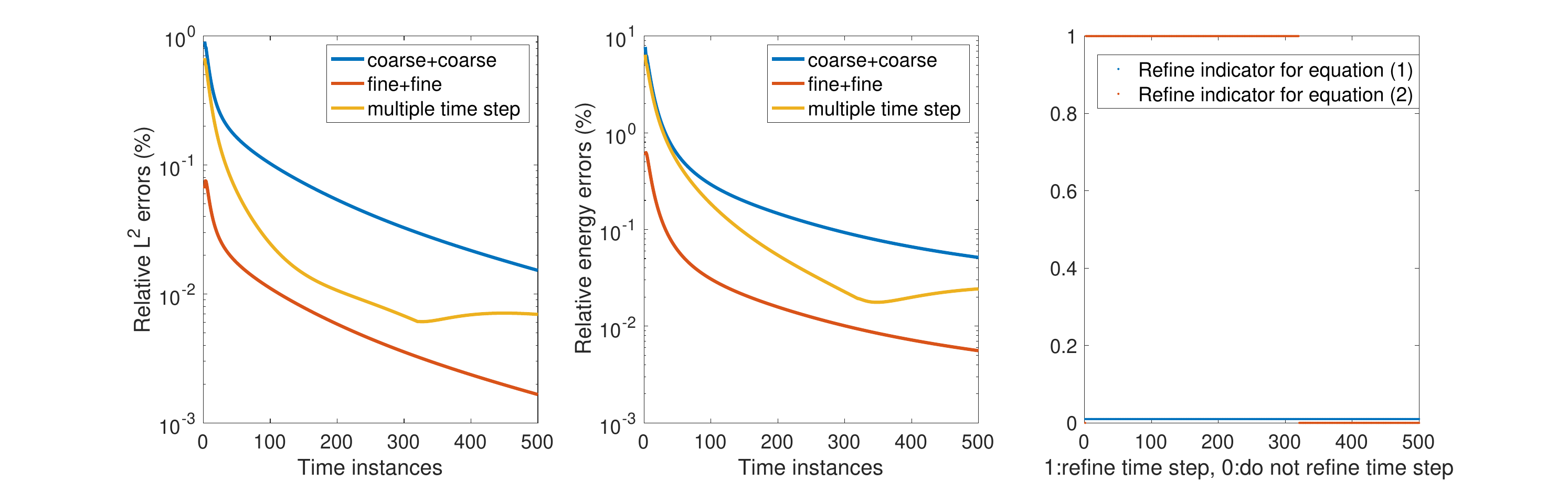}
    \caption{\rev{Example 1, using type 2 error indicators. $\delta_1 = 1.5\times 10^{-4}$, $\delta_2 = 5\times 10^{-6}$. Left and middle: error history, right: refinement history for two equations. The number of refined steps for the first equation is 0, for the second equation is 319. The mean $L^2$ error is 0.0311\%, and the energy error is 0.2370\%.}}
    \label{fig:errs_ex1_type2} %eta1=eta2=5e-3
\end{figure}

\section*{Acknowledgments}

\bibliographystyle{siam} %plain
\bibliography{references}

\end{document}